\documentclass[11pt, final]{article}
\usepackage{amsmath}
\usepackage{amsfonts}
\usepackage{makeidx}
\usepackage{amsthm}
\usepackage{mathtools}
\usepackage{amssymb}
\usepackage{bm}
\usepackage{cite}
\usepackage{comment}
\usepackage[english]{babel}
\usepackage{dblaccnt}
\usepackage{accents}
\usepackage{graphicx}
\usepackage{psfrag}
\usepackage{subfig}
\usepackage{color}
\usepackage{float}
\usepackage{amscd}
\usepackage[mathscr]{euscript}
\usepackage{lipsum}
\usepackage{epstopdf}
\usepackage{algorithm}
\usepackage{algpseudocode}
\usepackage[left=2.5cm,top=2cm,bottom=2cm,right=2.7cm]{geometry}

\newtheorem{defi}{Definition}[section]
\newtheorem{lemma}[defi]{Lemma}
\newtheorem{theorem}[defi]{Theorem}

\newcommand{\R}{\mathbb{R}}
\newcommand{\dsp}{\displaystyle}
\renewcommand{\div}{{\mathrm{div}}}
\newcommand{\supp}{{\mathrm{supp}}}
\newcommand{\loc}{\mathrm{loc}}

\graphicspath{{pics/}}

\title{Direct sampling methods for isotropic and anisotropic scatterers with point source measurements}

\author{Isaac Harris\thanks{Department of Mathematics, Purdue University, West Lafayette, IN 47907, USA; (harri814@purdue.edu, nguye686@purdue.edu) } \and Dinh-Liem Nguyen\thanks{Department of Mathematics, Kansas State University, Manhattan, KS 66506 USA; (dlnguyen@ksu.edu).} \and Thi-Phong Nguyen\footnotemark[1]}

\date{}

\begin{document}

\maketitle

\begin{abstract}
\noindent In this paper, we consider the inverse scattering problem for recovering either an isotropic or anisotropic scatterer from the measured scattered field initiated by a point source. We propose two new imaging functionals for solving the inverse problem. The first one employs a `far-field' transform to the data which we then use to derive and provide an explicit decay rate  for the  imaging functional. In order to analyze the behavior of this imaging functional we use the factorization of the near field operator as well as the Funk-Hecke integral identity. 
For the second imaging functional the Cauchy data is used to define the functional and its behavior is analyzed using the Green's identities. 
Numerical experiments are given in two dimensions for both isotropic and anisotropic scatterers.
\end{abstract}

\noindent{\bf Key words:}  inverse scattering, direct sampling method, near-field data, factorization method, isotropic and anisotropic scatterers \\

\noindent{\bf AMS subject classification:} 	35R30, 78A46

\section{Introduction} 
Here we will consider the inverse shape problem for recovering either an isotropic or anisotropic scatterer from the measured scattered field initiated by a point source. There are  many applications where one wishes to uses acoustic scattering waves to detect hidden structures in a given medium. This comes up in medical imaging and non-destructive testing. In many applications one has little to no a prior information of the scatterer that one wishes to reconstruct. Because of this {\it qualitative methods} (otherwise known as non-iterative or direct methods) have been used to solve multiple inverse shape problems in scattering as well as other imaging modalities. In this manuscript, we will study two new direct sampling imaging functionals to solve the inverse problem. Direct sampling methods are fast and stable algorithms to reconstruct the scattering object with little to no a priori information. 
They have been also investigated  in some studies under the name of orthogonality sampling method, see, e.g.~\cite{Potth2010,HarrisNguyen2020}.   These  sampling methods have been studied in great detail for far-field measurements for multiple inverse scattering problems, see~\cite{Potth2010, Ito2012, Ito2013, Liu2017, Nguye2019, HarrisNguyen2020} and references therein.  Direct sampling methods also have other variants \cite{Harri2019,dsm-fm} that are connected to the factorization method (see \cite{kirschbook} for details for details). 
One of the drawbacks of these sampling methods is that the analysis is typically incomplete and requires full-aperture data. There has been some work done in \cite{dsm-limap,dsm-dr} to justify these methods for limited-aperture data.

 Although direct sampling methods have been studied for many inverse scattering problems using far-field data  there is little to no investigation for the case of near-field data.  
 The imaging functionals are also studied for both  isotropic and anisotropic scatterers.  
In order to construct our imaging functionals and analyze their behavior we will need to develop suitable factorizations for the data operator for the measured scattered field initiated by a point source (see for e.g. \cite{Harris-Rome,nf-fm-isotropic,Shixu}). The factorization method was initially introduced in \cite{fm-paper} for far-field data but has been extended to many other models, see for e.g. \cite{kirschbook, Cakon2019,heat-fm}. 
Also we derive  explicit decay rates for our imaging functionals by two ways. The first is by applying a `far-field' transform to the measured scattered field and by appealing to the Funk-Hecke integral identity with the decay of the Bessel functions to prove the bounds. Another method for deriving the decay rate is to use the measured Cauchy data 
and the second Green's identity 

The rest of the paper is organized as follows. In Section \ref{sec:Iso} we begin by rigorously defining the scattering by an isotropic scatterer. Then we derive and analyze two new direct sampling imaging functionals. To do so, we first develop a suitable factorization of our measurement operator as well as discuss the Dirichlet-to-Far-Field operator. The  Dirichlet-to-Far-Field operator is critical to analyze the behavior for one of the imaging functionals. Then, in Section \ref{aniso} we similarly will consider the case for anisotropic scatterers. Here we show that the imaging functional with the `far-field' transform can also be used to recover anisotropic scatterers with similar resolution analysis. Lastly, in Section \ref{numerics} we provide a detailed numerical study of recovering both isotropic and anisotropic scatterers using the imaging functionals developed in the previous section.

\section{Isotropic Inverse Scattering Problem}\label{sec:Iso}
In this section, we will study two direct sampling methods for recovering a scatterer from point source measurements. To this end, let  $D \subset \R^d$ (for $d=2,3$) be a bounded simply connected open set with Lipschitz boundary and denote the inhomogeneous scattering medium. Outside of the scatterer we will assume that the medium is homogeneous with refractive index normalized to one. Denote by $q(x)\in L^{\infty}(\R^d)$ the contrast of the scattering medium with respect to free space such that $q = 1$ in $\R^d \setminus \overline{D}$ where $D=$ supp$(q)$. We consider the scattering by the incident point source generated at $y \in \Gamma$ where $\Gamma$ is a class $C^2$ curve/surface,  given by 
\begin{equation}\label{fund-solu}
	u^i(\cdot \,, y) = \Phi(\cdot \,, y) = \left\{ 
		\begin{array}{cl}
			\frac{\text{i}}{4} H^{(1)}_{0}(k|\cdot - y|) &\quad \text{if} \quad d = 2, \\[1.5ex]
			\displaystyle \quad\quad  \frac{\text{e}^{\text{i}k|\cdot - y|}}{4\pi|\cdot - y|} &\quad \text{if} \quad d = 3
		\end{array}
	\right.
\end{equation} where $k > 0$ is the wave number and $H^{(1)}_{0}$ is the first kind Hankel function of order zero. We let  $\Omega \in \R^d$  be a known bounded domain with boundary equal to the measurement/source boundary $\Gamma$ such that $D \subset \Omega$ and $\text{dist}(\Gamma , D)>0$.
The scattered field $u^s(\cdot \, , y)$ associated with the incident field $u^i(\cdot \, , y)$ solves the following equation
\begin{equation}
\label{eq:usca}
\left\{
	\begin{array}{lc}
		\Delta u^s(\cdot\, , y) + k^2(1+q) u^s(\cdot \, , y) = -k^2qu^i(\cdot \, , y), \quad \text{in} \quad \R^d \\[1.5ex]
	\dsp	\lim_{r \to \infty} r^{(d-1)/2}\left(\frac{\partial}{\partial r} - \text{i}k \right)u^s(\cdot \,, y) = 0,\quad r = |x|,
	\end{array}
\right.
\end{equation} 
where the second equation of \eqref{eq:usca} indicates the Sommerfeld radiation condition and holds uniformly over all angular directions.  If $\R^d \setminus \overline{D}$ is connected and $\Im(q) \geq 0$, then the scattering problem \eqref{eq:usca} is well-posed in $H^1_{\loc}(\R^d)$ for any $y \in \Gamma$ (see Chapter 8 of \cite{CK3}). It is well known that problem \eqref{eq:usca} can be written equivalently to the Lippmann-Schwinger equation (see \cite{CK3})
\begin{equation}
	u^s(x, y) = k^2 \int_{D} \Phi(x; \cdot) \Big(q u(\cdot\, ,y)\Big)  \mathrm{d}A. \label{LS-equ}
\end{equation} 
where \, $u(\cdot \, , y) = u^{s}(\cdot \, , y) + u^{i}(\cdot \, , y)$ is the total field. 

For the {\bf inverse problem}, we aim to recover the domain $D$ from the measured scattered field $u^s(x, y)$ for all $(x,y)\in \Gamma \times \Gamma$. {\color{black} In general, the measurement and source boundary need not be the same.} Two direct sampling imaging functionals will be analyzed for the case of an isotropic scatterer in this section. To begin, we first define the near-field operator $N: L^2(\Gamma) \longrightarrow L^2(\Gamma)$ given by  
\begin{equation}
\label{def:open}
	(Ng)(x)= \int_{\Gamma} u^s(x ,y)\, g(y)  \mathrm{d}s(y) \quad \forall \; x\in \Gamma. 
\end{equation} 
In order to develop a new imaging functional we first derive a factorization of the near-field operator. Here we develop a different factorization than what is given in \cite{nf-fm-isotropic}. Arguing similarly as in Section 4.2 of \cite{kirschbook} we can factorize the near-field operator $N$ using the following operators. Define the single layer potential $S: L^2(\Gamma) \longrightarrow L^2(D)$ such that 
\begin{equation}
(Sg)(x) = \int_{\Gamma} \Phi(x, y) g(y)  \mathrm{d}s(y), \quad \forall \; x\in D \label{S-operator}
\end{equation} 
and its dual operator $S^{\top}: L^2(D) \longrightarrow L^2(\Gamma)$  given by the volume potential 
\begin{equation}
(S^{\top} \varphi)(z) =  \int_{D} \Phi(x \, , z) \varphi(x)  \mathrm{d}x,  \quad \forall \; z \in \Gamma.
\end{equation}
Here the dual operator is understood with respect to the $L^2$ dual-product such that 
$$\langle  \varphi , Sg  \rangle_{L^2(D)} =\langle  S^{\top} \varphi , g  \rangle_{L^2(\Gamma)} \quad \text{for all } \quad g\in L^2(\Gamma) \textrm{ and } \varphi \in L^2(D)$$
where $\langle \phi \, , \psi \rangle_{L^2} = (\phi  , \overline{\psi} )_{L^2}$. Notice, that by linearity of the scattering problem that if the incident field $u^i$ is replaced by $Sg$ in \eqref{eq:usca} then we have that 
$$x \longmapsto \int_{\Gamma} u^s(x , y) g(y) \text{d}s(y) \quad  \forall\, x \in \R^d$$
is the corresponding scattered field.
We also define the bounded operator $T: L^2(D) \longrightarrow L^2(D)$ by
\begin{equation}
Tf =  \dsp k^2 q \big(w + f\big)|_{D}, \label{T-define}
\end{equation} 
where $w \in H^1_{\loc}(\R^d)$ solves the scattering equation \eqref{eq:usca} with $u^i(\cdot \, , y)$ is replaced by any arbitrary $f \in L^2(D)$ such that  
\begin{equation}
\label{eq:def:opet}
\left.
	\begin{array}{lc}
		\Delta w + k^2(1+q) w = -k^2 q f \quad \text{in} \quad \R^d \\[1.5ex]
		\dsp \lim_{r \to \infty} r^{(d-1)/2}\left(\frac{\partial}{\partial r} - \text{i}k \right)w= 0,\quad r = |x|.
	\end{array}
\right.
\end{equation} 
The boundedness of the operator $T$ comes from the well-posedness of \eqref{eq:def:opet} and the assumptions on the refractive index.
Notice, that for each $g \in L^2(\Gamma)$ we can denote by $f= Sg$ then by the well-posedness we can conclude that  
$$w(x)= \int_{\Gamma} u^s(x ,y)\, g(y)  \mathrm{d}s(y) \quad \forall \; x\in \R^d.$$
Now, since equation \eqref{eq:def:opet} is equivalent to the Lippmann-Schwinger equation \eqref{LS-equ} we have that 
$$	w(x) = k^2 \int_{D} \Phi(x ,\cdot ) \Big(q (w + Sg)\Big) \mathrm{d}A  \quad \forall \; x\in \R^d. $$
We notice that this implies that $w|_{\Gamma}$ is equivalent to $S^{\top} \, T \, Sg$ for any $x \in \Gamma$ by the definition of the operators, given above. This proves that the near-field operator has the factorization 
$$N g = S^{\top} \, T \, S g \quad \text{ for any} \quad g \in L^2(\Gamma)$$ 
since $w|_{\Gamma}$ is equal to $Ng$. 

\subsection{Direct Sampling Method with `Far-Field' Transformation}
We now develop the theory for a new direct sampling method for our inverse problem. The main idea is to increase the decay property of the imaging functional by implementing a so-called `Far-Field' transformation of the data. In similar direct sampling methods developed in the literature for the measurements of the scattered field given by point sources (see for e.g. \cite{Chen_2013,Ito2013}) one can only say that the imaging functionals will approximately decay as $\text{dist}(z , D)\to \infty$ where $z$ is the sampling point in $\R^d$. This is due to the use of the Helmholtz-Kirchhoff integral identity (see Lemma 3.2 of \cite{Chen_2013}). For the case of far-field measurements the decay  rate is given by the Funk-Hecke integral identity (see Lemma 2.7 of \cite{Liu2017}). Therefore, we can use the asymptotic decay rate of the Bessel functions for large arguments to develop the  explicit decay rate  as in \cite{Harri2019,HarrisNguyen2020}.  

Now that we have derived a suitable factorization of the near-field operator we  wish to rigorously define the Dirichlet-to-Far-Field transformation that will be used to define the imaging functional. To this end, let $v\in H^1_{\loc}(\R^d \setminus \overline{\Omega})$ be the unique solution to 
\begin{equation}
\label{eq:opeq}
\begin{array}{rll}
(\Delta +k^2) v  =  0 \quad \text{in}  \quad \R^d \setminus\overline{\Omega} \quad \text{ with } \quad v|_\Gamma & = & h  \\[1.5ex]
 \dsp \lim_{r \to \infty} r^{(d-1)/2}\left(\frac{\partial}{\partial r} - \text{i}k \right)v & = &0
\end{array} 
\end{equation}
for any $h \in H^{1/2}(\Gamma)$ (see Theorem 3.11 in \cite{CK3} for the well-posedness of this problem). Therefore, we have that $v$ has the expansion 
$$v(x)=  \frac{\text{e}^{\text{i}k|x|}}{|x|^{(d-1)/2}} \left\{v^{\infty}(\hat{x}) + \mathcal{O} \left( \frac{1}{|x|}\right) \right\}\; \textrm{  as  } \;  |x| \to \infty$$
where $\hat x:=x/|x|$ and $v^{\infty}(\hat x) $ is the corresponding far-field pattern (see Chapter 2 of \cite{CK3}). 
Being motivated by previous works on direct sampling methods for far-field measurements we define the Dirichlet-to-Far-Field transformation $\mathcal{Q}: H^{1/2}(\Gamma) \longrightarrow L^2(\mathbb{S}^{d-1})$ such that
\begin{align}
(\mathcal{Q} h)(\hat{x}) = v^{\infty}(\hat{x}), \quad \forall \; \hat{x} \in \mathbb{S}^{d-1} \label{Q-operator}
\end{align}
where $\mathbb{S}^{d-1}$ denotes the unit circle/sphere. 

We will now connect the factorization of the near-field operator with the Dirichlet-to-Far-Field transformation in order to derive an imaging functional to recover the scatterer $D$. Notice, that the volume potential 
$$ v= \int_{D}\Phi(x, \cdot) \varphi(x)  \mathrm{d}{x} \quad  \textrm{for any } \quad \varphi \in L^2(D)$$
solves \eqref{eq:opeq} where $h = S^{\top} \varphi$. It is well-known that $S^{\top}: L^2(D) \longrightarrow H^{3/2}(\Gamma)$ by Theorem 8.2 of \cite{CK3} and the Trace Theorem. By applying to the asymptotic expansion of the fundamental solution (see for e.g. \cite{CCH-book,CK3}) we have that 
\begin{equation}
\label{eq:qst}
(\mathcal{Q} \, S^{\top} \varphi) (\hat{y}) = v^{\infty}(\hat{y}) = \int_{D} \text{e}^{-\text{i}k \hat{y}\cdot x} \varphi(x)  \mathrm{d}{x},
\end{equation} for all $\hat{y} \in \mathbb{S}^{d-1}$. The above operator is the well known adjoint operator to the Herglotz wave function $ H: L^2( \mathbb{S}^{d-1}) \longrightarrow L^2(D)$ given by
$$H g = \int_{\mathbb{S}^{d-1}} \text{e}^{\text{i}k  x \cdot \hat{y}} g(\hat{y})  \mathrm{d} s(\hat{y}) \quad \text{ where } \quad (Hg , \varphi)_{ L^2(\mathbb{S}^{d-1})} = (g , H^* \varphi)_{L^2(D)} $$
for all $g \in L^2(\mathbb{S}^{d-1})$ and $\varphi \in L^2(D)$. Therefore, equation \eqref{eq:qst} can be written as 
$$\mathcal{Q} \, S^{\top} \varphi = H^*	\varphi \quad \textrm{for any } \quad \varphi \in L^2(D)$$
and by  appealing to the factorization $$N = S^{\top}TS \quad \text{ implies that} \quad \mathcal{Q} N =H^*TS.$$ 

Next, we will build an indicator function for the sampling method based on operator $\mathcal{Q}  N$ instead of $N$. Since $\Omega$ is known one can compute $\mathcal{Q}$ independently in order to construct the operator for solving the inverse problem. Also, note that $\mathcal{Q}$ can be computed in a multitude of ways. One can use boundary integral operators for the Helmholtz equation to derive the solution operator for \eqref{eq:opeq}(see for e.g. Chapter 3 of \cite{CK3}). Then by appealing to asymptotic formula for fundamental solution one can derive a formula for the operator $\mathcal{Q}$. Later in this section we will derive a formula for $\mathcal{Q}$ when $\Gamma$ is the boundary of a ball centered at the origin. Also note that, $\mathcal{Q}$ is independent of the underline scattering problem which implies that this operator can be used to study other problems in inverse scattering for near-field data sets.

In order to define our new imaging functional, we now introduce for each sampling point $z \in \Omega$ the functions $\phi^{(1)}_z(x) \in L^{2}(\Gamma)$ and $\phi^{(2)}_z(\hat{y}) \in L^2(\mathbb{S}^{d-1})$ such that 
$$\phi^{(1)}_z(x) = \overline{\Phi(x, z)}\quad  \textrm{ and } \quad  \phi^{(2)}_z(\hat{y}) = \text{e}^{-\text{i}k z \cdot \hat{y}} \quad \text{for any} \; z \in \Omega.$$
{\bf The imaging functional via Far-Field transform:} We can now define and analyze the direct sampling methods imaging functional via a far-field transform. To this end, we let $z \in \Omega$ be a sampling point and recall $N: L^2(\Gamma) \longrightarrow L^2(\Gamma)$ and $\mathcal{Q}: H^{1/2}(\Gamma) \longrightarrow L^2(\mathbb{S}^{d-1})$ as defined in \eqref{def:open} and \eqref{Q-operator} respectively. Then the imaging functional via a far-field transform is given by 
\begin{equation}
I_{\text{FF}}(z) = \Big|\Big(\mathcal{Q} N \phi^{(1)}_z, \phi^{(2)}_z \Big)_{L^2(\mathbb{S}^{d-1})}\Big|. \label{dsm1}
\end{equation} 
From the factorization in the previous section we can now develop the resolution analysis of the proposed imaging functional. This will give an explicit decay rate of the $I_{\text{FF}}(z)$ as dist$(z,D) \to \infty$ which will validate plotting $I_{\text{FF}}(z)$ to recover the scatterer $D$. The fact that the decay rate is explicit is due to the far-field transform and the Funk-Hecke integral identity.

In order to provide the decay rate we need the following result(i.e. the Helmholtz-Kirchhoff integral identity). 
\begin{lemma}\label{lem:auxi1}
Let $S: L^2(\Gamma) \longrightarrow L^2(D)$ be as defined in \eqref{S-operator} then we have that 
$$ \left(S\phi^{(1)}_z \right)(x) = \frac{1}{k}\Big(\Im \Phi(x, z) + \omega(x, z)\Big), \quad \forall x \in D, \;  z \in \Omega$$
where $\|\omega( \cdot \, , z) \|_{H^1(D)} \leq C$ where the constant $C$ depends on $\Gamma$ but is independent of  $z \in \Omega$.
\end{lemma}
\begin{proof}
This is a simple consequence of Lemma 3.2 in \cite{Chen_2013}.
\end{proof}
Recall, that by definition of the fundamental solution given in \eqref{fund-solu}
$$\Im \Phi(x, z) = \left\{\begin{array}{lr} \frac{1}{4} J_0(k | x - z|) \, \, & \quad \text{if} \quad d = 2  \,,\\
 				&  \\
\frac{1}{4\pi} j_0(k | x - z|) & \quad \text{if} \quad d = 3 \,.
 \end{array} \right. $$
where $J_0(t)$ and $j_0(t)$ are the zeroth Bessel and spherical Bessel function of the first kind. This implies that the leading order term in 
$$\| S\phi^{(1)}_z \|_{L^2(D)}= \mathcal{O}(1) \quad \text{ as dist}(z,D) \to \infty.$$ 
Without augmenting that data one can't get an explicit rate of decay from the Helmholtz-Kirchhoff  identity due to the presence of the $\omega(x, z)$ which is only given to be bounded as dist$(z,D) \to \infty$.

\begin{theorem}\label{thm1}
Let the imaging functional $I_{\text{FF}}(z)$ be defined by \eqref{dsm1}. Then  for every $z \notin D$ 
$$I_{\text{FF}}(z)  =  \mathcal{O} \left(\text{dist}(z,D)^{(1-d)/2}  \right)\quad \text{ as } \,\,\, \text{dist}(z,D) \to \infty \; \textrm{ for } \; d = 2, 3.$$
\end{theorem}
\begin{proof}
In order to prove the claim, recall that we have the factorization 
$$ \mathcal{Q} N g= H^*TS g \quad \text{for all } \,\, \in L^2(\Gamma)$$ 
Now, we can clearly see that by the definition of the adjoint that 
\begin{align*}
\left(\mathcal{Q} N \phi^{(1)}_z, \phi^{(2)}_z \right)_{L^2(\mathbb{S}^{d-1})}  &= \, \big(H^*TS \phi^{(1)}_z, \phi^{(2)}_z \big)_{L^2(\mathbb{S}^{d-1})} \\[1.ex]
	 & = \, \big( TS \phi^{(1)}_z, H \phi^{(2)}_z \big)_{L^2(D)}
\end{align*}
We now recall the Funk-Hecke integral identity, given by  
\[ \left(H \phi^{(2)}_z \right)(x) =  \int_{\mathbb{S}^{d-1}}  \mathrm{e}^{- \mathrm{i} k {(z - x)}  \cdot \hat{y} } \, \mathrm{d}s(\hat{y} ) =\left\{\begin{array}{lr} 2\pi J_0(k | x - z|) \, \, & \quad \text{if} \quad d = 2,\\
 				&  \\
4\pi j_0(k | x - z|) & \quad \text{if} \quad d = 3.
 \end{array} \right. \]
With this and the boundedness of the operator $T: L^2(D) \longrightarrow L^2(D)$ we have that there is a constant $C>0$ independent of $z$ such that 
\begin{align*}
\Big|\big(\mathcal{Q} N \phi^{(1)}_z, \phi^{(2)}_z  \big)_{L^2(\mathbb{S}^{d-1})}\Big| & \leq \, C \|S \phi^{(1)}_z\|_{L^2(D)} \|H \phi^{(2)}_z\|_{L^2(D)} \\
	&\leq C \| H \phi^{(2)}_z  \|_{L^2(D)} \Big( \|\Im \Phi(\cdot, z) \|_{L^2(D)} + \|\omega(\cdot, z) \|_{L^2(D)} \Big)
\end{align*}
 Since $J_0(t)$ has a decay rate of $t^{-1/2}$ and $j_0(t)$ has a decay rate of  $t^{-1}$ as $t \to \infty$ we have that   
 $$\| H \phi^{(2)}_z  \|_{L^2(D)} \leq C \left( \text{dist}(z,D)^{(1-d)/2} \right) \quad \text{ as } \,\,\, \text{dist}(z,D) \to \infty.$$ 
Together with  $\|\omega (\cdot, z) \|_{H^1(D)} \leq C$ for all $z \in \Omega$ we obtain the estimate 
$$\Big|\big(\mathcal{Q}N\phi^{(1)}_z, \phi^{(2)}_z  \big)_{L^2(\mathbb{S}^{d-1})}\Big| \leq C \text{dist}(z,D)^{(1-d)/2} \Big( 1 + \text{dist}(z,D)^{(1-d)/2} \Big) \quad \text{ as } \,\,\, \text{dist}(z,D) \to \infty$$
 for some constant $C > 0$ independent of $z \in \Omega$. This is obviously equivalently to the estimate 
$$\Big|\big(\mathcal{Q}N\phi^{(1)}_z, \phi^{(2)}_z  \big)_{L^2(\mathbb{S}^{d-1})} \Big| \leq C\left( \text{dist}(z,D)^{(1-d)/2} \right) \quad \text{ as } \,\,\, \text{dist}(z,D) \to \infty$$
which proves the claim.
\end{proof}

From this we have that the imaging functional $I_{\text{FF}}(z)$ should decay as the sampling point moves away from the scatterer.  Also, note that since $N: L^2(\Gamma) \longrightarrow L^2(\Gamma)$ is known from the measured data and $\mathcal{Q}: H^{1/2}(\Gamma) \longrightarrow L^2(\mathbb{S}^{d-1})$ can be precomputed without a priori knowledge of $D$, this is a fast and simple method for recovering the scatterer. 

For simplicity, we will consider the case when $\Gamma$ is the boundary of a  disk centered at the origin with radius $R$ fixed. In this case, $\Omega=B(0;R)$ is a disk where the unknown scatterer $D \subset B(0;R)$. Since we are considering the case when $\Gamma = \partial B(0;R)$ we can derive a formula for computing 
$$\mathcal{Q}: H^{1/2}(\Gamma) \longrightarrow L^2(\mathbb{S}^{d-1})$$ 
via separation of variables. Note that, similar calculation can also be done in three dimensions. Now, we identify $H^{1/2}(\Gamma)$ with the space $H^{1/2}(0,2\pi)$ and $L^2(\mathbb{S}^{d-1})$ with the space $L^2(0,2 \pi)$. Therefore, for all  $h(\theta) \in H^{1/2}(0,2\pi)$ can be written in terms of the Fourier series  
$$h(\theta) = \sum_{|m| = 0}^{\infty} {h}_{m} \text{e}^{\text{i}m\theta} \quad \text{ where } \quad  h_m =  \frac{1}{2\pi} \int\limits_0^{2\pi} h(\phi) \text{e}^{-\text{i} m \phi} \text{d} \phi \quad \text{ for all} \,\, m\in \mathbb{Z}.$$ 
One can easily check that the unique solution $v (r,\theta)$ where $r=|x|$ and $\theta = \text{Arg}(x_1+ \text{i} x_2)$ to equation \eqref{eq:opeq} can then be written as the series   
$$v(r,\theta) = \sum_{|m| = 0}^{\infty} {v}_{m}  H^{(1)}_{m}(kr) \text{e}^{\text{i}m\theta} \quad \text{ where } \quad  {v}_m = \frac{{h}_m}{H^{(1)}_{m}(kR)}  \quad \text{ for all} \,\, m\in \mathbb{Z}.$$
Here we let  $H^{(1)}_{m}$ denote the first kind Hankel function of order $m$.
Therefore, by appealing to the asymptotic relationship of the Hankel function 
$$ H^{(1)}_{m}(kr) = \sqrt{\frac{2}{\pi k r}} \text{e}^{\text{i}kr-\text{i}m\pi/2 -\text{i}\pi/4 } + \mathcal{O}(r^{-3/2}) \quad \text{ as } \quad r \to \infty$$
we have that 
\begin{align}
v^{\infty}(\theta)  = \sum_{|m| = 0}^{\infty} \frac{ \sqrt{2}(1-\text{i}) }{\sqrt{k\pi}} \frac{{h}_{m}}{H^{(1)}_{m}(kR)} \text{e}^{\text{i}m (\theta-\pi/2)}
 & = \frac{1}{2\pi} \int\limits_0^{2\pi}\frac{ \sqrt{2}(1-\text{i}) }{\sqrt{k\pi}} \sum_{|m| = 0}^{\infty} \frac{\text{e}^{\text{i}m(\theta - \phi - \pi/2)}}{H^{(1)}_{m}(kR)}  h(\phi) \text{d}{\phi} \nonumber \\[1.ex]
 & = \frac{1-\text{i}}{\pi\sqrt{2k\pi}} \int\limits_0^{2\pi} \sum_{|m| = 0}^{\infty} \frac{\text{e}^{\text{i}m(\theta - \phi - \pi/2)}}{H^{(1)}_{m}(kR)} h(\phi) \text{d}{\phi}. \label{Q-formula}
\end{align}
This implies that the operator $\mathcal{Q}: H^{1/2}(0,2\pi) \longrightarrow L^2(0,2 \pi)$ can be written as
$$ (\mathcal{Q} h)(\theta) = \int\limits_0^{2\pi} K(\theta, \phi) h(\phi) \text{d}{\phi} \quad \text{ with kernel function } \quad K(\theta, \phi) = \frac{1-\text{i}}{\pi\sqrt{2k\pi}} \sum_{|m| = 0}^{\infty} \frac{\text{e}^{\text{i}m(\theta - \phi - \pi/2)}}{H^{(1)}_{m}(kR)}.$$
In numerically evaluating the imaging functional $I_{\text{FF}}(z)$ given by \eqref{dsm1} we will need to approximate  $\mathcal{Q}$ with a discretized version of the operator that converges in norm. To this end, we consider the truncated series representation denoted by the operator truncated series for some $M \in \mathbb{N}$ 
$$ (\mathcal{Q}_M h)(\theta) = \int\limits_0^{2\pi} K_M (\theta, \phi) h(\phi) \text{d}{\phi} \quad \text{ with kernel function } \quad K_M (\theta, \phi) = \frac{1-\text{i}}{\pi\sqrt{2k\pi}} \sum_{|m| = 0}^{M} \frac{\text{e}^{\text{i}m(\theta - \phi - \pi/2)}}{H^{(1)}_{m}(kR)}.$$
In the next result, we prove convergence of the operator $\mathcal{Q}_M$ to the Dirichlet-to-Far-Field transformation in the operator norm from $H^{1/2}(0,2\pi) \longrightarrow L^2(0,2 \pi)$. Here we use the definition of the $H^{p}(0,2\pi)$-norm via the Fourier coefficients. 

\begin{lemma}\label{Q-converge}
Let $\mathcal{Q}: H^{1/2}(0,2\pi) \longrightarrow L^2(0,2 \pi)$ be the Dirichlet-to-Far-Field transformation defined by \eqref{Q-formula} and $\mathcal{Q}_M: H^{1/2}(0,2\pi) \longrightarrow L^2(0,2 \pi)$ be the truncated series for some $M \in \mathbb{N}$. Then we have norm-convergence with convergence rate given by
$$\|\mathcal{Q} - \mathcal{Q}_{M}\|_{H^{1/2}(0,2\pi) \mapsto L^2(0,2 \pi)} =\mathcal{O}\left( \frac{1}{ 2^M} \right) , \quad \text{as} \quad M \longrightarrow \infty.$$
\end{lemma}
\begin{proof}
In order to prove the claim, we see that 
\begin{align*}
\left[\mathcal{Q} - \mathcal{Q}_{M} \right] h &=\int\limits_{0}^{2\pi} \big[K(\theta,\phi)-K_M (\theta,\phi)\big] h(\phi) \text{d}{\phi} \\
&=\sum_{|m| = M+1}^{\infty} \frac{1-\text{i}}{\pi\sqrt{2k\pi}} \frac{{h}_{m}}{H^{(1)}_{m}(kR)} \text{e}^{\text{i}m (\theta-\pi/2)}.
\end{align*} 
Now, we compute the $L^2(0,2\pi)$-norm of the function $\left[\mathcal{Q} - \mathcal{Q}_{M} \right] h$ given by 
$$\|(\mathcal{Q} - \mathcal{Q}_{M})h\|^2_{L^2(0,2\pi) }\leq C \sum_{|m| = M+1}^{\infty}  \frac{ |{h}_{m}|^2}{\big| H^{(1)}_{m}(kR) \big|^2 }$$
In order to estimate, we now use the facts that  $H^{(1)}_{-m}(t) = (-1)^m H^{(1)}_{m}(t)$ for all $m \in \mathbb{Z}$ as well as the asymptotic result (see for e.g. \cite{bessel-webpage})
$$ -\text{i} H^{(1)}_{m}(t) \sim \sqrt{\frac{2}{\pi m}} \left(\frac{\text{e}t}{2m}\right)^{-m} \quad \text{ as } \quad m \to \infty.$$ 
That implies, there exists some constant $C$ independent from $m$ such that
\begin{align*}
 \|(\mathcal{Q} - \mathcal{Q}_{M})h\|^2_{L^2(0,2\pi)}&\leq C \sum_{|m| = M+1}^{\infty} |m| |{h}_{m}|^2  \left( \frac{\text{e}kR}{2m} \right)^{2m} \\
 	&= \frac{C}{2^{2M}} \sum_{|m| = M+1}^{\infty} |m| |{h}_{m}|^2  \left( \frac{\text{e}kR}{m} \right)^{2m}
\end{align*}
and we can clearly see that the root test implies that the sequence 
$$ \left( \frac{\text{e}kR}{m} \right)^{2m} \longrightarrow 0 \quad \text{ as } \quad  m \longrightarrow \infty \quad \text{ and is therefore bounded}.$$
This implies that 
$$\|(\mathcal{Q} - \mathcal{Q}_{M})h\|^2_{L^2(0,2\pi) }\leq \frac{C}{2^{2M}} \sum_{|m| = 0}^{\infty} \big( 1+ |m|^2\big)^{1/2} |{h}_{m}|^2  =  \frac{C}{2^{2M}}  \| h \|^2_{H^{1/2}(0,2\pi)}$$
which proves the claim. 
\end{proof}

In practice, the geometric convergence rate in Lemma \ref{Q-converge} means that we can approximate the operator $\mathcal{Q}$ by the truncated series $\mathcal{Q}_M$. Therefore, do to the fast convergence one does not need to keep many terms in the series. This will allow one to evaluate the approximate operator $\mathcal{Q}_M$ more efficiently in numerically solving the inverse problem.  

Remark: For the case when $\Gamma \neq  \partial B(0;R)$ we have that the solution $v$ to \eqref{eq:opeq} can be written using a single layer potential as in \cite{kirschbook}. Therefore, the far-field pattern can be obtained using the asymptotic of the fundamental solution. This is similar to obtaining the factorization of the far-field operator for the scattering by a sound-soft scatterer, see Chapter 2 of \cite{kirschbook} for details.  

\subsection{Direct Sampling Method with Cauchy Data}\label{DSMcd}
In this section, we develop another imaging functional for the recovering the scatterer $D$ from point source measurements. Here we will assume that we have the Cauchy data  given by $u^s(x,y)$ and $\partial_{\nu} u^s(x , y)$ for all  $(x,y) \in \Gamma \times {\Gamma}$.  We again note that the measurement and source boundary need not be the same. Again, the goal is to derive an explicit decay rate as dist$(z,D) \to \infty$ where $z\in \Omega$ is the sampling point. As we will see by using the Cauchy data we along with Green's identities  we can derive an explicit decay rate which can not be done by only taking the near-field data $u^s(x,y)$ which has not been studied in the literature.

 We will now define the imaging functional for the case of given Cauchy data.  Now, provided that both $u^s(x,y)$ and $\partial_{\nu} u^s(x,y)$ is given for all  $(x,y) \in \Gamma \times {\Gamma}$,  we define a new imaging functional
\begin{align}
I_{\text{CD}}(z) = \int_{\Gamma} \left | \int_\Gamma \partial_{\nu} \overline{ \Phi(x,z)} u^s(x,y) - \overline{\Phi(x,z)} \partial_{\nu} u^s(x,y)  \mathrm{d}s(x)\right |^\rho  \mathrm{d}s(y) \label{dsm2}
\end{align}
where for $\rho >0$ is a fixed constant. Here the parameter $\rho$ can be used to sharpen the resolution when recovering the scatterer numerically(see for e.g. \cite{Liu2017}). In order to analyze the $I_{\text{CD}}(z)$ we will write the imaging functional in terms of the operator $T: L^2(D) \longrightarrow L^2(D)$ defined in \eqref{T-define}. 

 \begin{theorem}
Let the imaging functional $I_{\text{CD}}(z)$ be defined by \eqref{dsm2}. Then for every $z \in \Omega$ 
$$I_{CD}(z) = \int_{\Gamma} \left |\int_D \frac{1}{2\text{i}} \Im \Phi(z,\cdot)  T \Phi(  \cdot \, ,y)  \text{d}A \right |^\rho \text{d}s(y)$$
where $u^s$ is the unique solution to \eqref{eq:usca}.
\end{theorem}
\begin{proof}
To begin, recall that by the Lippmann-Schwinger equation \eqref{LS-equ} we have that the Cauchy data can be written as the volume integrals 
$$ u^s(x,y) = k^2 \int_{D} \Phi(x , \cdot) \Big(q u( \cdot \, ,y)\Big)  \mathrm{d}A \quad \text{and} \quad \partial_{\nu} u^s(x,y) = k^2 \int_{D} \partial_{\nu}\Phi(x , \cdot) \Big(q u( \cdot \, , y)\Big)  \mathrm{d}A$$
where the total field $u( \cdot \, ;y) = u^{s}(\cdot \, , y) + \Phi(\cdot \, , y)$. Now, by the auxiliary scattering problem \eqref{eq:def:opet} and the definition of the operator $T$ we obtain the identity 
$$ u^s(x,y) =  \int_{D} \Phi(x , \cdot ) T \Phi( \cdot \,,y)  \mathrm{d}A \quad \text{and} \quad \partial_{\nu} u^s(x,y) = \int_{D} \partial_{\nu}\Phi(x , \cdot ) T \Phi(\cdot \,  ,y)  \mathrm{d}A.$$
Notice, that we have
\begin{align*}
& \int_\Gamma \partial_{\nu} \overline{ \Phi(x,z)} u^s(x,y) - \overline{\Phi(x,z)} \partial_{\nu} u^s(x,y)  \text{d}s(x)  \\
 &\hspace{0.4in} =  \int_\Gamma  \left[ \overline{\partial_{\nu}  \Phi(x,z)}  \int_{D} \Phi(x , \cdot) T \Phi(  \cdot \, ,y)  \mathrm{d}A - \overline{ \Phi(x,z)}   \int_{D} \partial_{\nu}  \Phi(x , \cdot) T \Phi(  \cdot \, ,y)  \mathrm{d}A\right]  \text{d}s(x) \\
 & \hspace{0.4in} =  \int_\Gamma  \int_D  \left[ \partial_{\nu}  \overline{ \Phi(x,z)} \Phi(x,\cdot)-  \overline{ \Phi(x,z)} \partial_{\nu} \Phi(x,\cdot)\right]  T \Phi(  \cdot \,,y) \text{d}A \,  \text{d}s(x)\\
 & \hspace{0.4in}= \int_D    \left[ \int_\Gamma \partial_{\nu}  \overline{ \Phi(x,z)} \Phi(x,\cdot) -  \overline{ \Phi(x,z)} \partial_{\nu} \Phi(x,\cdot)  \text{d}s(x) \right] T \Phi(  \cdot \, ,y)  \text{d}A.
\end{align*}
Using Green's second identity, we can prove that for all $z\in \Omega$ 
$$\Phi(z,\cdot) - \overline{\Phi(z,\cdot)} =  \int_\Gamma   \left[ \partial_{\nu}  \overline{ \Phi(x,z)} \Phi(x,\cdot) -  \overline{ \Phi(x,z)} \partial_{\nu} \Phi(x,\cdot) \right] \text{d}s(x).$$
Indeed, this is verified by first appealing to equations
$$\Delta_x \Phi(x,\cdot) + k^2\Phi(x,\cdot) = -\delta(x-\cdot) \quad \text{and } \quad \Delta_x \overline{\Phi(x,z)} + k^2\overline{\Phi(x,z)} = -\delta(x-z). $$
Then, multiplying the first and second equations by $ \overline{\Phi(x,z)}$ and $\Phi(x,\cdot)$, respectively and integrating over $\Omega$ we obtain
$$\int_{\Omega} \Delta_x \Phi(x,\cdot)  \overline{\Phi(x,z)} - \Delta_x \overline{\Phi(x,z)} \Phi(x,\cdot) \, \text{d}x=  - \overline{\Phi(\cdot,z)} +  \Phi(z,\cdot).$$
Then the desired identity follows from Green's second identity 
and the symmetry of the fundamental solution. Therefore,
we obtain 
$$\int_\Gamma \partial_{\nu} \overline{ \Phi(x,z)} u^s(x,y) - \overline{\Phi(x,z)} \partial_{\nu} u^s(x,y)  \text{d}s(x) = 
\int_D \frac{1}{2\text{i}} \Im \Phi(z,\cdot)  T \Phi(  \cdot \, ,y)  \text{d}A$$
and substituting this identity in the imaging functional we have that 
$$I_{CD}(z) = \int_{\Gamma} \left |\int_D \frac{1}{2\text{i}} \Im \Phi(z,\cdot)  T \Phi(  \cdot \, ,y)  \text{d}A \right |^\rho \text{d}s(y)$$
which proves the claim. 
\end{proof}

We recall that, 
$$\Im \Phi(\cdot, z) = \left\{\begin{array}{lr} \frac{1}{4} J_0(k | \cdot - z|) \, \, & \quad \text{if} \quad d = 2  \,,\\
 				&  \\
\frac{1}{4\pi} j_0(k | \cdot - z|) & \quad \text{if} \quad d = 3 \,.
\end{array} \right. $$
From the above result we have that the imaging functional $I_{\text{CD}}(z)$ should be maximal on the interior of the scatterer $D$ and takes small values on the exterior of $D$. Now, that we have the equivalent representation of $I_{CD}(z) $ with respect to the operator  $T$ as defined in \eqref{T-define} we can show the decay rate  for the imaging functional.

 \begin{theorem}\label{thm2}
Let the imaging functional $I_{\text{CD}}(z)$ be defined by \eqref{dsm2}. Then for every $z \notin D$ 
$$I_{\text{CD}}(z)  =  \mathcal{O} \left( \text{dist}(z,D)^{(1-d)\rho/2} \right) \quad \text{ as } \,\,\, \text{dist}(z,D) \to \infty \; \textrm{ for } \; d = 2, 3.$$
\end{theorem}
\begin{proof}
To begin, we first recall that 
$$I_{CD}(z) = \int_{\Gamma} \left |\int_D \frac{1}{2\text{i}} \Im \Phi(z,\cdot)  T \Phi(  \cdot \, ,y)  \text{d}A \right |^\rho \text{d}s(y)$$
and the fact that $T: L^2(D) \longrightarrow L^2(D)$ as a bounded linear operator. Therefore, by the Cauchy-Schwartz inequality we have that 
$$I_{CD}(z) \leq C  \|\Im \Phi(z,\cdot) \|_{L^2(D)}^\rho \int_{\Gamma} \| \Phi(\cdot \, , y) \|^\rho_{L^2(D)} \text{d}s(y).$$
Since, we have assumed that $\text{dist}(\Gamma , D)>0$ we can have that 
$$\int_{\Gamma} \| \Phi(\cdot \, , y) \|^\rho_{L^2(D)} \text{d}s(y) \quad \text{ if a fixed constant independent of $z$.}$$
Then, the estimate for the imaging functional becomes 
$$I_{CD}(z) \leq C  \|\Im \Phi(z,\cdot) \|_{L^2(D)}^\rho $$
and we again use the fact that  $J_0(t)$ has a decay rate of  $t^{-1/2}$ and $j_0(t)$ has a decay rate of  $t^{-1}$ as $t \to \infty$, which proves the claim.
\end{proof}
 
We also note that, if only $u^s(x,y)$ is measured for all  $(x,y) \in \Gamma \times {\Gamma}$, where  $\Gamma$  is the boundary of a ball with large radius,
and the measurement curve/surface $\Gamma$ is sufficiently far away from the scatterer then $\partial_{\nu} u^s(x,y) \approx \text{i}ku^s(x,y)$ on $\Gamma$ by the radiation condition. 
The imaging functional can be approximated by 
$$I_{\text{CD}}^{\mathrm{far}}(z)
 = \int_{\Gamma} \left | \int_\Gamma \big[\partial_{\nu}  \overline{\Phi(x,z)}- \text{i}k \overline{\Phi(x,z)} \big] u^s(x,y) \text{d}s(x)\right |^\rho \text{d}s(y).$$ 
One would expect that $I_{\text{CD}}(z) \approx I_{\text{CD}}^{\mathrm{far}}(z)$ for any $z \in \Omega$ which would imply that plotting $I_{\text{CD}}^{\mathrm{far}}(z)$ could also be used to recover the scatterer provided one does not have or can not compute $\partial_{\nu} u^s(x,y)$ for any $(x,y) \in \Gamma \times {\Gamma}$ provided $\Gamma$ is the boundary of a ball with large radius.

\section{Anisotropic Inverse Scattering Problem}\label{aniso}
One of the main advantages for using a qualitative reconstruction method is that the same algorithm will work for multiple scattering problems. In this section, we will show that the imaging functional via a far-field transform $I_{\text{FF}}(z)$ defined in \eqref{dsm1}  as well as $I_{\text{CD}}(z)$ can also be used to solve that inverse scattering problem for an anisotropic scatterer. This shows the novelty of this direct sampling method as being a simple and robust algorithm for recovering scatterers from point source measurements. In this problem, we are interested in determining the support of an anisotropic inhomogeneous medium, that is characterized by constitutive parameters $Q$ and $q$. 

We will assume, that the matrix valued parameter satisfies that $Q(x)$ is the zero matrix for all $x\notin D$ where $Q \in C^{1}  \left( D, \mathbb{C}^{d \times d} \right)$. Furthermore, assume that for any vector $\xi \in \mathbb{C}^d$
$$\overline{\xi}\cdot \Re \left( I+Q(x) \right) \xi\geq Q_{min} |\xi|^2>0 \quad \text{and}  \quad \overline{\xi}\cdot \Im \left(  Q(x) \right)\xi \leq 0 \quad \text{for a.e. $x \in D$} $$
where $I$ is the identity matrix. Here we assume $q$ satisfies the same assumptions as in the previous section. Similarly, to illuminate the medium by an incident field $u^i(\cdot \,, y) = \Phi(\cdot \,, y)$  generated by a point source $y \in \Gamma$ then the scattered field $u^s \in H^1_{\loc}(\R^d)$ satisfies 
\begin{equation}
\label{eq:uscaA}
\left\{
	\begin{array}{lc}
		\div \Big( (I+Q) \nabla u^s(\cdot \,, y)\Big) + k^2(1+q)  u^s(\cdot \,, y) = - \div \Big(Q \nabla u^i(\cdot \,, y)\Big) - k^2 qu^i(\cdot \,, y), \quad \text{in} \quad \R^d \\[1.5ex]
	\dsp	\lim_{r \to \infty} r^{(d-1)/2}\left(\frac{\partial}{\partial r} - \text{i}k \right)u^s(\cdot \,, y) = 0.
	\end{array}
\right.
\end{equation}  where $Q$ and $q$ are the contrast for the anisotropic scattering medium. The well-posedness of \eqref{eq:uscaA} is given in Theorem 1.38 of \cite{CCH-book}. We again have an equivalent to the Lippmann-Schwinger equation (see for e.g. \cite{CCH-book})
\begin{equation}
\label{eq:Lipp:aniso}
	u^s(x, y) = -\div \int_{D} \Phi(x, \cdot) \Big(Q \nabla u(\cdot \, ,y)\Big) \text{d}A + k^2 \int_{D} \Phi(x,\cdot) \Big(q u(\cdot\, , y)\Big) \text{d}A. 
\end{equation} 
where again the total field is given by  $u(\cdot \,, y) = u^{s}(\cdot \,, y)+ u^{i}(\cdot \,, y)$ and $D = \supp(Q) \cup \supp(q)$. The measurements for the inverse problem are $u^s(x , y)$ for $(x,y) \in \Gamma \times \Gamma$. The analysis of the direct sampling method to handle this problem can be done similarly to the isotropic case with  some modifications on the factorization of the near-field operator. This change is governed by the volume integral in the Lippmann-Schwinger equation \eqref{eq:Lipp:aniso}. The factorization of the near-field operator for an anisotropic scatterer was studied in \cite{Harris-Rome} where non-physical incident fields are used. In order to derive a suitable factorization we use similar arguments as in Chapter 2 of \cite{CCH-book}. 

We begin by reintroducing  near-field operator $N: L^2(\Gamma) \longrightarrow \; L^2(\Gamma)$ such that 
$$(N g)(x) =  \int_{\Gamma} u^s(x , y) g(y) \text{d}s(y), \quad \forall x \in \Gamma .$$
Just as in the previous section, we have that by linearity of the scattering problem we have that for single layer potential $S: L^2(\Gamma) \longrightarrow L^2(D)$ 
$$(Sg)(x) = \int_{\Gamma} \Phi(x, y) g(y)  \mathrm{d}s(y), \quad \forall \; x\in D$$
then for $Sg$ taken to be the incident field in \eqref{eq:uscaA} we have that 
$$x \longmapsto \int_{\Gamma} u^s(x , y) g(y) \text{d}s(y) \quad  \forall\, x \in \R^d$$
is the corresponding scattered field. Now, motivated by the anisotropic Lippmann-Schwinger equation \eqref{eq:Lipp:aniso} we define the following operators. 
It is well-known that $S: L^2(\Gamma) \longrightarrow H^1_{\loc}(\R^d)$ see \cite{McLean} Theorem 6.11. So we can define the bounded linear operator
\begin{equation}
\mathcal{S}_Q: L^2(\Gamma) \longrightarrow \;  [L^2(D)]^{d+1} \quad \text{ such that } \quad \mathcal{S}_Q g= \Big(\nabla Sg|_{D} \, , \, Sg|_{D}\Big). \label{S-aniso}
\end{equation}
We have that the corresponding dual operator $\mathcal{S}_Q^{\top}: [L^2(D)]^{d+1}  \longrightarrow L^2(\Gamma)$ is given by
\begin{equation}
\mathcal{S}_Q^{\top}( \bm{\varphi},\psi) = \div \int_{D} \Phi(z, y)  \bm{\varphi} (y) \text{d}y + \int_{D} \Phi(z, y) \psi(y) \text{d}y , \quad \forall \; z \in \Gamma.\label{Sdual-aniso}
\end{equation} 
To continue, just as in the isotropic case let $w \in H^1_{\loc}(\R^d)$ be the unique solution to the problem 
\begin{equation}
\label{eq:wA}
\left.
	\begin{array}{lc}
	\div \Big( (I+Q) \nabla w\Big) + k^2(1+q) w = - \div (Q  \bm{\varphi} ) - k^2 q  \psi \quad \text{in} \quad \R^d \\[1.5ex]
	\dsp	\lim_{r \to \infty} r^{(d-1)/2}\left(\frac{\partial}{\partial r} - \text{i} k \right)w = 0.
	\end{array}
\right.
\end{equation} 
for any given $(\bm{\varphi},\psi) \in [L^2(D)]^{d+1}$ where we assume that $\bm{\varphi} \in[L^2(D)]^d$ and $\psi \in L^2(D)$(see for e.g. \cite{CCH-book} for the well-posedness of \eqref{eq:wA}). Now, to continue we define the bounded operator 
\begin{equation}
\begin{array}{rl}
\mathcal{T}: [L^2(D)]^{d+1} \longrightarrow [L^2(D)]^{d+1} \quad \text{where } \quad \mathcal{T}( \bm{\varphi} ,\psi) = \Big( - Q \big( \bm{\varphi} +\nabla w \big)|_{D} \, , \,k^2q \big(w + \psi \big)|_{D}\Big).
\end{array} 
\end{equation} 

Similar to the  isotropic case, let $( \bm{\varphi} ,\psi)  =\left(\nabla Sg|_{D}, Sg|_{D}\right)= \mathcal{S}_Q g$ in equation \eqref{eq:wA}, which implies that the solution can be written as
$$w(x) = -\div \int_{D} \Phi(x, \cdot) \Big(Q (\nabla w + \nabla Sg)\Big) \text{d}A + k^2 \int_{D} \Phi(x,\cdot) \Big(q (w + Sg) \Big) \text{d}A \quad  \forall\, x \in \R^d. $$
Therefore,  we can see that $w$ given above is the scattered field for  \eqref{eq:uscaA} when $u^i$ is taken to be $Sg$. This implies that $w|_{\Gamma}$ is equal to the near-field operator $Ng$ for any $g$ which when combined with the definition of the above operators gives that 
$$Ng=\mathcal{S}_Q^{\top} \mathcal{T}\mathcal{S}_Q g \quad \text{ for any} \quad g \in L^2(\Gamma).$$ 
Now that we have this factorization of the near-field operator for anisotropic scatterers, we can show that the direct sampling methods imaging functional via a far-field transform can be used to recover the scatterer in this case as well.

\subsection{Direct Sampling Method with `Far-Field' Transformation}
Now we can proceed with the decay rate as as $\text{dist}(z , D)\to \infty$ for the direct sampling method via a far-field transform. We will show that the decay of the imaging functional is equivalent to the result in Theorem \ref{thm1} for isotropic scatterers. To begin, we recall the Dirichlet-to-Far-Field transformation $\mathcal{Q}: H^{1/2}(\Gamma) \longrightarrow L^2(\mathbb{S}^{d-1})$ as defined in \eqref{Q-operator} as well as the imaging functional 
$$I_{\text{FF}}(z) = \Big|\Big(\mathcal{Q} N \phi^{(1)}_z, \phi^{(2)}_z \Big)_{L^2(\mathbb{S}^{d-1})}\Big|$$ 
where 
 $$ \phi^{(1)}_z(x) = \overline{\Phi(x, z)}\quad  \textrm{for any } \quad  \phi^{(2)}_z(\hat{y}) = \text{e}^{-\text{i}k z \cdot \hat{y}} \quad \text{for any} \; z \in \Omega.$$
Again, one of the advantages is that the imaging functional does not change for the anisotropic scattering problem.

In order to develop the resolution analysis for anisotropic scatterers, we need to compute the operator $\mathcal{Q}\mathcal{S}_Q^{\top}$ for any $(\bm{\varphi},\psi) \in [L^2(D)]^{d+1}$. Notice, that the volume potential 
$$ v= \int_{D}\Phi(x, \cdot) \varphi(x)  \mathrm{d}{x} \quad  \textrm{for any } \quad \varphi \in L^2(D)$$
maps $L^2(D) \longrightarrow H^{2}(\Omega)$ as a bounded linear operator by \cite{CK3} Theorem 8.2 and standard elliptic regularity \cite{evans}. This implies that Range$(\mathcal{S}_Q^{\top})\subset H^{1/2}(\Gamma)$ and therefore the operator $\mathcal{Q}\mathcal{S}_Q^{\top}$ is well defined. By again applying to the asymptotic expansion of the fundamental solution we have that
$$\mathcal{Q}\mathcal{S}_Q^{\top} (\bm{\varphi},\psi) = \int_{D} \left[-\text{i}k \hat{y} \cdot \bm{\varphi}(x) + \psi(x) \right] \text{e}^{-\text{i}kx \cdot \hat{y}} \text{d}x \quad \forall \, \hat{y} \in \mathbb{S}^{d-1}.$$
From this, we recall the Herglotz wave function $H: L^2( \mathbb{S}^{d-1}) \longrightarrow L^2(D)$ given by
$$H g = \int_{\mathbb{S}^{d-1}} \text{e}^{\text{i}k  x \cdot \hat{y}} g(\hat{y})  \mathrm{d} s(\hat{y}) $$
and define  
$$\mathcal{H}_Q: L^2( \mathbb{S}^{d-1}) \longrightarrow [L^2(D)]^{d+1} \quad \text{such that} \quad \mathcal{H}_Qg = \Big(\nabla H g|_{D} \, , \, Hg|_{D}\Big).$$
Therefore, by the analysis in Chapter 2 of \cite{CCH-book} we have that 
$$\mathcal{H}_Q^* (\bm{\varphi},\psi) = \int_{D} \left[-\text{i}k \hat{y} \cdot \bm{\varphi}(x) + \psi(x) \right] \text{e}^{-\text{i}kx \cdot \hat{y}} \text{d}x  \quad \text{ which implies that} \quad \mathcal{Q}\mathcal{S}_Q^{\top} = \mathcal{H}_Q^*.$$ 
From this, we have obtained that $\mathcal{Q}N=\mathcal{H}_Q^* \mathcal{T}\mathcal{S}_Q$. We now have all we need to provide the resolution analysis result. 

\begin{theorem}\label{dsm3}
Let the imaging functional $I_{\text{FF}}(z)$ be defined by \eqref{dsm1}. Then for every $z \notin D$ 
$$I_{\text{FF}}(z)  =  \mathcal{O} \left( \text{dist}(z,D)^{(1-d)/2} \right) \quad \text{ as } \,\,\, \text{dist}(z,D) \to \infty \; \textrm{ for } \; d = 2, 3$$
where  $N: L^2(\Gamma) \longrightarrow L^2(\Gamma)$ is the near-field operator corresponding to \eqref{eq:uscaA}.
\end{theorem}

\begin{proof} The proof is similar to the case of an isotropic scatterer. 
Recall, that we have the factorization $\mathcal{Q}N=\mathcal{H}_Q^* \mathcal{T}\mathcal{S}_Q$ which implies that 
$$\big(\mathcal{Q}N  \phi^{(1)}_z \, , \, \phi^{(2)}_z \big)_{L^2(\mathbb{S}^{d-1})}  = \, \Big(\mathcal{H}^*_Q \, \mathcal{T} \, \mathcal{S}_Q  \phi^{(1)}_z  \, , \,  \phi^{(2)}_z \Big)_{L^2(\mathbb{S}^{d-1})}  = \, \Big(  \mathcal{T} \, \mathcal{S}_Q \phi^{(1)}_z  \, , \,  \mathcal{H}_Q \,\phi^{(2)}_z \Big)_{[L^2(D)]^{d+1}} $$
In order to estimate the inner-product given above, we notice that 
$$\mathcal{S}_Q  \phi^{(1)}_z = \Big(\nabla S\phi^{(1)}_z|_{D} \, , \, S\phi^{(1)}_z |_{D} \Big)$$
and by Lemma \ref{lem:auxi1}
$$\left(S\phi^{(1)}_z \right)(x) = \frac{1}{k}\Big(\Im \Phi(x, z) + \omega(x, z)\Big).$$
This gives the estimate 
$$\| \mathcal{S}_Q  \phi^{(1)}_z\|_{[L^2(D)]^{d+1}} \leq C \| \Im \Phi(\cdot, z)\|_{H^1(D)} + \|\omega (\cdot, z)\|_{H^1(D)}$$
where the $H^1(D)$-norm of $\omega( \cdot \, , z)$ is bounded with respect to $z$. Now, we consider 
$$\mathcal{H}_Q  \phi^{(2)}_z = \left(\nabla H  \phi^{(2)}_z|_{D} \, , \, H \phi^{(2)}_z|_{D}\right)$$
and recall the Funk-Hecke integral identity 
\[ \left(H \phi^{(2)}_z \right)(x) =  \int_{\mathbb{S}^{d-1}}  \mathrm{e}^{- \mathrm{i} k {(z - x)}  \cdot \hat{y} } \, \mathrm{d}s(\hat{y} ) =\left\{\begin{array}{lr} 2\pi J_0(k | x - z|) \, \, & \quad \text{if} \quad d = 2,\\
 				&  \\
4\pi j_0(k | x - z|) & \quad \text{if} \quad d = 3.
 \end{array} \right. \]
Therefore,  by the fact that $J_0(t)$ (as well as it's derivatives) has a decay rate of  $t^{-1/2}$ and $j_0(t)$ (as well as it's derivatives) has a decay rate of  $t^{-1}$ as $t \to \infty$ we obtain the estimates
$$\| \mathcal{S}_Q  \phi^{(1)}_z\|_{[L^2(D)]^{d+1}} \leq C\Big( 1 + \text{dist}(z,D)^{(1-d)/2} \Big)$$
as well as 
$$\|\mathcal{H}_Q  \phi^{(2)}_z \|_{[L^2(D)]^{d+1}} \leq C \text{dist}(z,D)^{(1-d)/2}$$
as $\text{dist}(z,D) \to \infty$. By combining the above estimates with the Cauchy-Schwartz inequality and the boundedness of $\mathcal{T}$ proves the claim.
\end{proof}

From Theorem \ref{dsm3} we have that the imaging functional $I_{\text{FF}}(z)$ is a simple yet stable method for recovering either an isotropic or anisotropic scatterer from the measured scattered field. This is useful since in many practical applications you may not know a prior if the scatterer is isotropic or anisotropic, or at least may not have estimates on the coefficients. Since the matrix valued coefficient uniquely determined by the scattering data \cite{not-uniq} this give a simple method for solving the inverse shape problem of recovering $D$ without any estimates for the matrix value coefficient. 

\subsection{Direct Sampling Method with Cauchy Data}
In this section, we present the resolution analysis for the the direct sampling method with Cauchy data. This would give that both imaging functionals derived in Section \ref{sec:Iso} can be used to recover isotropic and anisotropic scatterers. Just as in the previous section we will see that similar analysis can be used to study that indicator 
$$I_{\text{CD}}(z) = \int_{\Gamma} \left | \int_\Gamma \partial_{\nu} \overline{ \Phi(x,z)} u^s(x,y) - \overline{\Phi(x,z)} \partial_{\nu} u^s(x,y)  \mathrm{d}s(x)\right |^\rho  \mathrm{d}s(y)$$ 
initially defined in \eqref{dsm2}. We wish to prove a similar result as in Theorem \ref{thm2} for the case of an anisotropic scatterer.

In order to develop the resolution analysis we consider $w \in H^1_{\loc}(\R^d)$ that is the solution to equation 
\begin{equation}
\label{eq:aniw}
\left\{
	\begin{array}{lc}
	\div \Big( (I+Q) \nabla w\Big) + k^2(1+q) w = - \div \Big(Q \nabla f\Big) - k^2 q f, \quad \text{in} \quad \R^d \\[1.5ex]
	\lim\limits_{r \to \infty} r^{(d-1)/2}\left(\frac{\partial}{\partial r} - \text{i}k \right)w = 0.
	\end{array}
\right.
\end{equation}
for any given $f \in H^1(D)$. This is motivated by anisotropic scattering problem \eqref{eq:uscaA} for the scattered field $u^s$. The well-posedness of \eqref{eq:aniw} is guaranteed by  Theorem 1.38 of \cite{CCH-book}. From equation \eqref{eq:aniw} we see that $w$ satisfies 
$$\Delta w + k^2 w = - \div \Big(Q \nabla (w+f) \Big) - k^2 q(w + f)$$
and by the Lippmann-Schwinger equation (see (8.13) of \cite{CK3}) we have that 
$$w(x) = \int_{D} \Phi(x,\cdot) \Big(\div \big( Q \nabla (w+f)  \big) + k^2q (w + f) \Big) \,  \text{d}A. $$
By the above representation of $w$ we define the operator 
\begin{align}
{T_Q} f = \div(Q  (\nabla w + \nabla f) + k^2q (w + f) \label{Tq}
\end{align}
such that $T_Q :  H^1(D) \longrightarrow [H^1(D)]'$ where $[H^1(D)]'$ is the dual space of $H^1(D)$ with respect to the $L^2(D)$ dual-product as given in the previous section. Indeed, notice that for any $g \in H^1(D)$ we have that 
\begin{align*} 
\langle   g , {T_Q} f   \rangle_{L^2(D)}  & =  \int_{D} \Big[\div(Q  (\nabla w + \nabla f) + k^2q (w + f) \Big]g \, \text{d}A\\
							  &= - \int_{D} Q \nabla (w+f)  \cdot \nabla g - k^2q (w + f) g \, \text{d}A
\end{align*}
where we have used Green's first identity along with the fact that supp$(Q)=D$. Then it is a simple application of the Cauchy-Schwartz inequality and the well-posedness of \eqref{eq:aniw} to obtain 
\begin{align}
\big| \langle   g , {T_Q} f   \rangle_{L^2(D)} \big| \leq C \| f \|_{H^1(D)} \| g \|_{H^1(D)} \quad \text{for all } \quad f \, , \, g\in H^1(D)   \label{T-dual}
\end{align}
where the constant $C>0$ depends on the contrasts $Q$ and $q$ in the set $D$. This implies that $T_Q :  H^1(D) \longrightarrow [H^1(D)]'$ is a bounded linear operator.

Notice, that by anisotropic scattering problem \eqref{eq:uscaA} for the scattered field $u^s$ can be written as
$$u^s(x,y) = \int_{D} \Phi(x,\cdot) \Big(\div(Q \nabla u(\cdot \,,y)) + k^2q u(\cdot\, , y)\Big) \text{d}A$$ 
where again the total field is given by  $u(\cdot \,, y) = u^{s}(\cdot \,, y)+ u^{i}(\cdot \,, y)$. Therefore, by the definition of the operator $T_Q$ \eqref{Tq} we have that 
$$u^s(x,y) =  \int_{D} \Phi(x , \cdot ) T_Q \Phi( \cdot \,,y)  \mathrm{d}A \quad \text{and} \quad \partial_{\nu} u^s(x,y) = \int_{D} \partial_{\nu}\Phi(x , \cdot ) T_Q \Phi(\cdot \,  ,y)  \mathrm{d}A$$
where the integrals are understood as a dual-pairing on $H^1(D) \times [H^1(D)]'$. We now have all we need to prove that $I_{\text{CD}}(z)$ has the same resolution for thee anisotropic scattering problem.

 \begin{theorem}
Let the imaging functional $I_{\text{CD}}(z)$ be defined by \eqref{dsm2}. Then for every $z \in \Omega$ 
$$I_{CD}(z) = \int_{\Gamma} \left |\int_D \frac{1}{2\text{i}} \Im \Phi(z,\cdot)  T_Q \Phi(  \cdot \, ,y)  \text{d}A \right |^\rho \text{d}s(y)$$
where $u^s$ is the unique solution to \eqref{eq:uscaA}.
\end{theorem}
\begin{proof}
The proof is analogous to what is presented in Section \ref{DSMcd}. 
\end{proof} 

Therefore, we now the similar decay rate for the imaging functional. 

 \begin{theorem}\label{dsm4}
Let the imaging functional $I_{\text{CD}}(z)$ be defined by \eqref{dsm2}. Then for every $z \notin D$ 
$$I_{\text{CD}}(z)  =  \mathcal{O} \left( \text{dist}(z,D)^{(1-d)\rho/2} \right) \quad \text{ as } \,\,\, \text{dist}(z,D) \to \infty \; \textrm{ for } \; d = 2, 3.$$
\end{theorem}
\begin{proof}
We first note that just as in the previous section we have that
$$I_{CD}(z) = \int_{\Gamma} \left |\int_D \frac{1}{2\text{i}} \Im \Phi(z,\cdot)  T_Q \Phi(  \cdot \, ;y)  \text{d}A \right |^\rho \text{d}s(y)$$
Therefore, by the estimate \eqref{T-dual} we have that 
$$I_{CD}(z) \leq C  \|\Im \Phi(z,\cdot) \|_{H^1(D)}^\rho \int_{\Gamma} \| \Phi(\cdot \, , y) \|^\rho_{H^1(D)} \text{d}s(y).$$
Arguing just as in Theorem \ref{thm2}, we have that the imaging functional is bounded above such that  
$$I_{CD}(z) \leq C  \|\Im \Phi(z,\cdot) \|_{H^1(D)}^\rho. $$
Again, we use the fact that $J_0(t)$ (as well as it's derivatives) has a decay rate of  $t^{-1/2}$ and $j_0(t)$ (as well as it's derivatives) has a decay rate of $t^{-1}$ as $t \to \infty$. Proving the claim.
\end{proof}

 From the results given in this and the previous sections gives reconstruction method for recovering scatterers for both Isotropic and Anisotropic scatterers via direct sampling methods. This is new in the since that we have developed the resolution analysis. This justifies that plotting the proposed indicators will recover the scatterer. We see that the imaging functional works for multiple types of scatterers which implies that one does not need a priori knowledge of the type of scatterer that is being recovered. 
 
\section{Numerical Examples }\label{numerics}

We now present several numerical examples in 2D to illustrate the performance of the proposed imaging functionals $I_{\text{FF}}$ and $I_{\text{CD}}$ for  reconstructing both isotropic and anisotropic scatterers. In all our examples. we choose $\Gamma$ to be the boundary of the circle centered at $0$ with radius $R=3$. In order to discretize the curve $\Gamma$ we take $m$ uniformly distributed points given by $y_j$.
The synthetic data is then generated by the $m$ point sources that are located at $y_j$ for $j=1, \cdots , m$. We first solve numerically the direct problem by the spectral solver (see \cite{numericspaper} for details). This method is based on writing the problem under the form of Lippmann-Schwinger equation and applying FFT to accelerate matrix-vector product. This corresponds to the data $ u^s(x_i , y_j)$ where $x_i \in \Gamma$ are uniformly distributed points for $i=1, \cdots , m$. In practice, the measured data is always given with noise. Therefore, we denote by $\delta$ the given noisy level and we define the noisy data as
$$u^{s,\delta} (x_i , y_j)= u^s(x_i , y_j) \left(1 + \delta E_{i,j}\right) \quad \text{and} \quad  \partial_{\nu} u^{s,\delta} (x_i , y_j) =  \partial_{\nu} u^{s} (x_i , y_j) \left(1 + \delta E_{i,j}\right)$$
where $E$ is the random matrix of size $m \times m$ where the Frobenius matrix norm $\|E\|_{F} = 1$. 

Now, in order to approximate the imaging functionals using the trapezoidal rule to approximate the integrals. In the functional $I_{\text{CD}}$ this corresponds to the 2D trapezoidal rule approximation for the integral over $\Gamma \times \Gamma$. For computing $I_{\text{FF}}$, we first discretized the near-field operator $N$ via the $m$ point trapezoidal rule on $\Gamma$ and Dirichlet-to-Far-Field transformation $\mathcal{Q}_M$ via a uniformly spaced $m$ point trapezoidal rule on $\mathbb{S}^{1}$. Then, the resulting inner-product is also discretized by the trapezoidal rule.  Here, we take $m = 100$, $\rho = 2$, $M = 20$ in all the following examples. Since direct sampling methods are known to be very robust against noise in the data, we take the noise level $\delta = 50\%$ in the following examples.

The wave number $k  = 8$ is fixed with parameters $\{Q,q\}$  taken to be either $\{Q_1,q_1\}$, $\{Q_2,q_2\}$,  and $\{Q_3,q_3\}$ in the scatterer with
\begin{equation*}
\{Q_1,q_1\} = \left\{\left[\begin{array}{cc}
0 & 0 \\
0 & 0
\end{array} \right], 0.3 \right\}, \; \{Q_2,q_2\} = \left\{\left[\begin{array}{cc}
0.3 & 0 \\
0 & 0.5
\end{array} \right], 0 \right\}, \; \{Q_3,q_3\}= \left\{ \left[\begin{array}{cc}
0.3 & 0 \\
0 & 0.5
\end{array} \right], 0.2 \right\}.
\end{equation*}
Notice that $\{Q_1,q_1\}$ corresponds to an isotropic scatterer where as $\{Q_2,q_2\}$,  and $\{Q_3,q_3\}$  corresponds to an anisotropic scatterer. We provided examples of reconstructing multiple scatterers for both cases. The geometry of the scatterers to be recovered are given in Figure \ref{scatterers}.
\begin{figure}[h!!!] 
\centering
\subfloat[One Scatterer]{\includegraphics[width=0.32\linewidth]{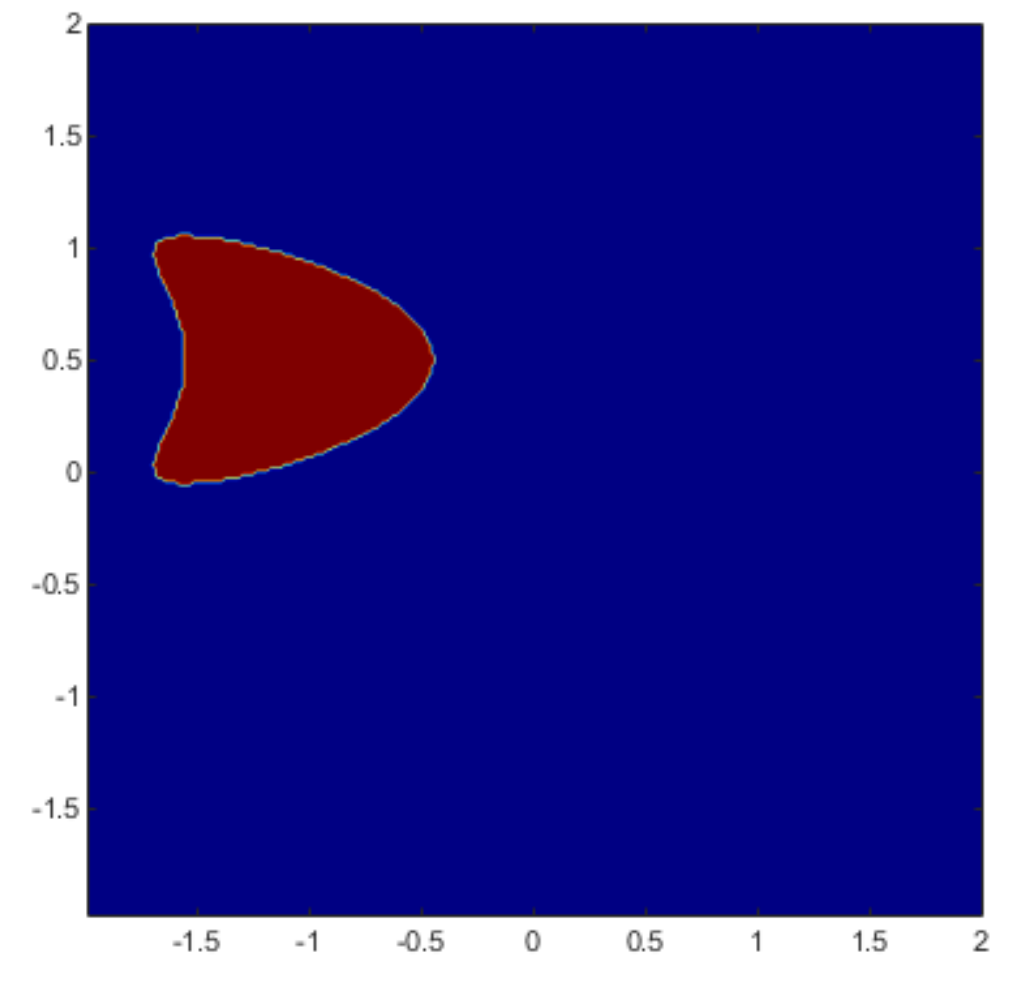}}
\subfloat[Two Scatterers]{\includegraphics[width=0.32\linewidth]{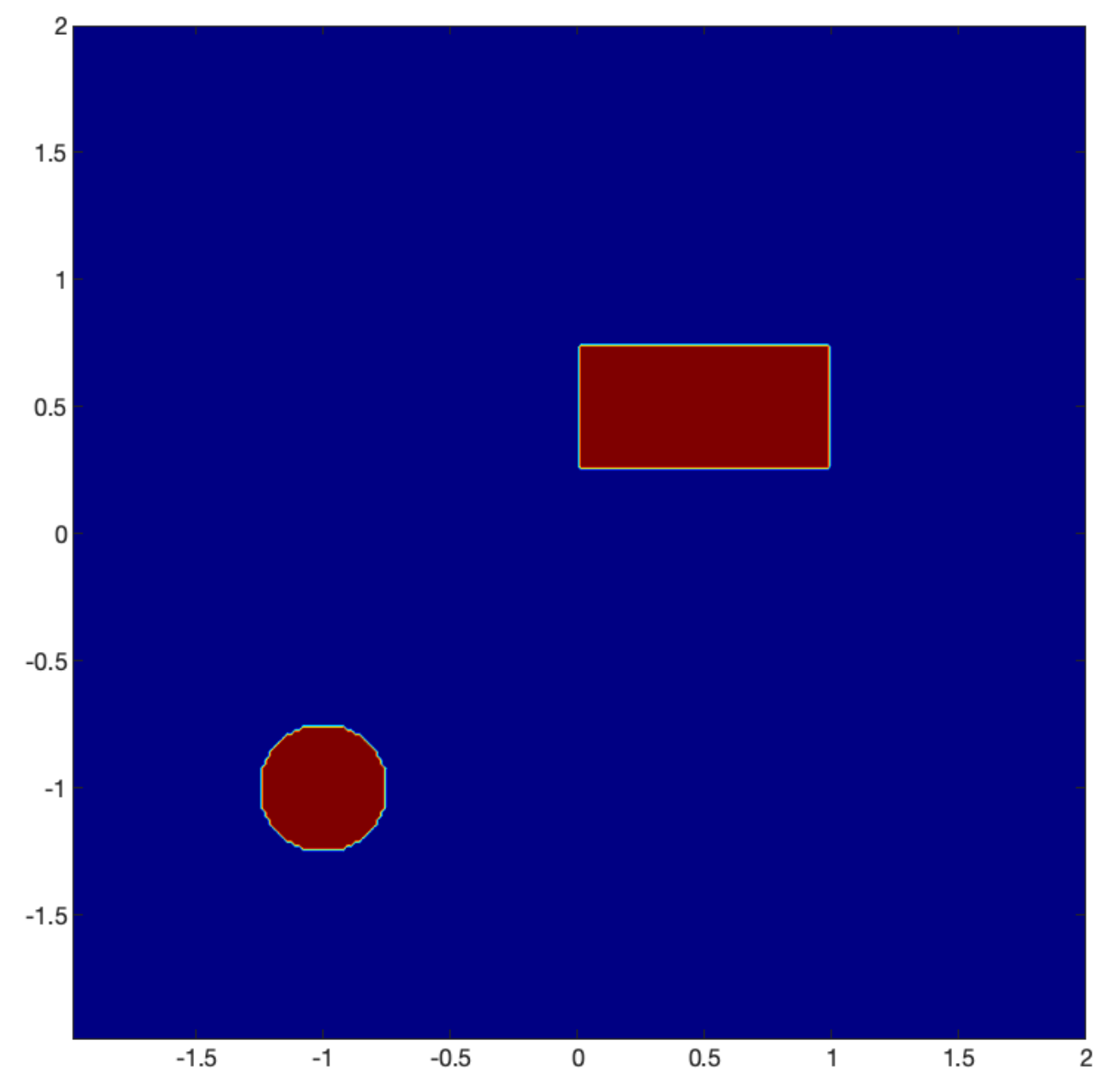}}
\subfloat[Three Scatterers]{\includegraphics[width=0.32\linewidth]{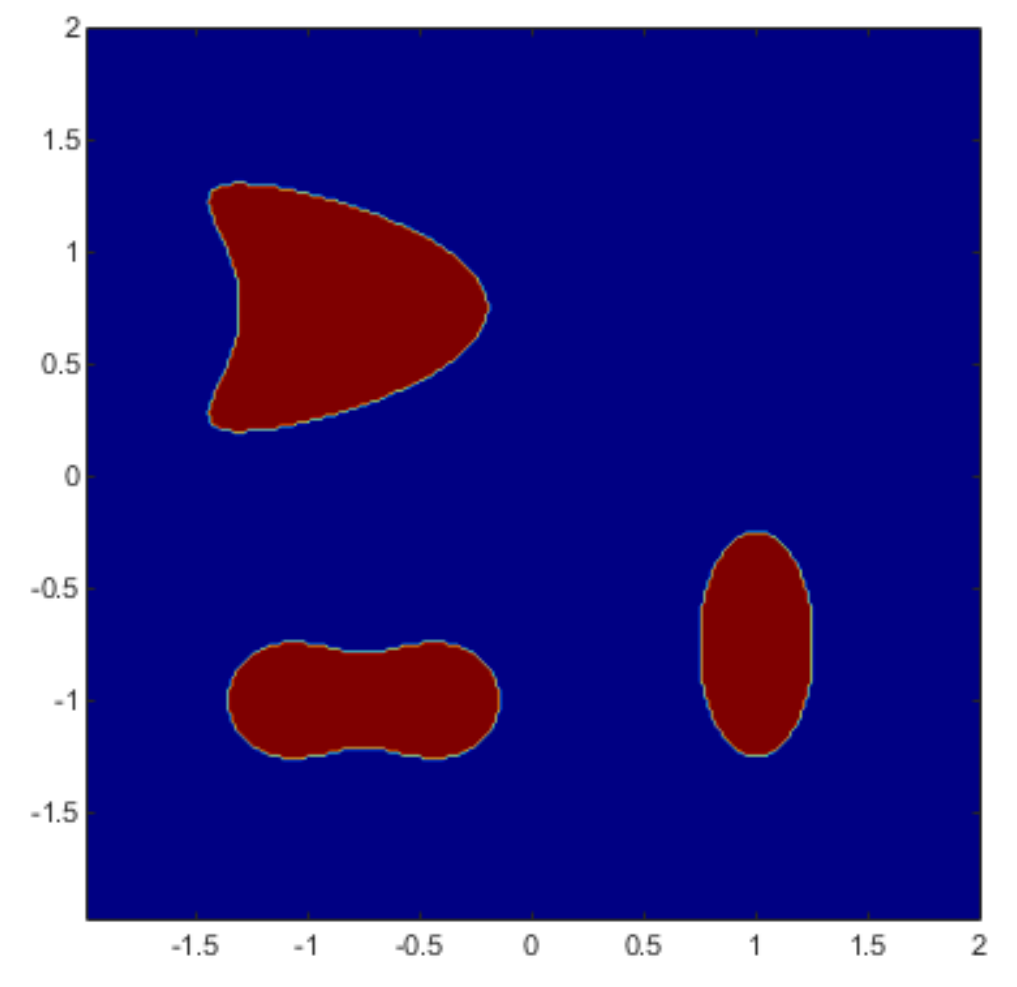}}
\caption{The exact geometry of the scatterer(s)}
\label{scatterers}
\end{figure}
Where we show that the imaging functionals can reconstruct multiple scatterers. In order to visualize the imaging functionals we provide the contour plot for $I_{\text{FF}} (z)$ and $I_{\text{CD}}(z)$ in the region $[-2,2]\times [-2,2]$. \\

\noindent{\bf Example 1: One Scatterer} \\
For our first numerical example we will consider recovering the kite shaped scatterer depicted in Figure \ref{scatterers} (a). We plot the imaging functionals to recover the kite-shaped scatterer $D$ where the boundary $\partial D=$kite which is given by 
\begin{equation*}
\begin{array}{ll}
\text{kite} & = \; \left\{(x_1, x_2)  \, : \, \begin{array}{ll}
x_1 &= -0.75 + 0.3\cos t - 1.84(0.75 + 0.3\sin t)^2, \\
x_2 &= 0.75 + 0.3\sin t ;
\end{array} \; 0 \leq t \leq 2\pi\right\}. 
\end{array}
\end{equation*}
The reconstruction for $I_{\text{FF}} (z)$ is given in Figure \ref{recon1FF} and $I_{\text{CD}}(z)$ is given in Figure \ref{recon1CD}, respectively. Here we consider the case of an isotropic and anisotropic scatterer. \\
\begin{figure}[ht]
\centering
\subfloat[$I_{\text{FF}}(z)\quad \text{with} \quad \{Q_1 ,q_1\}$]{\includegraphics[width=0.32\linewidth]{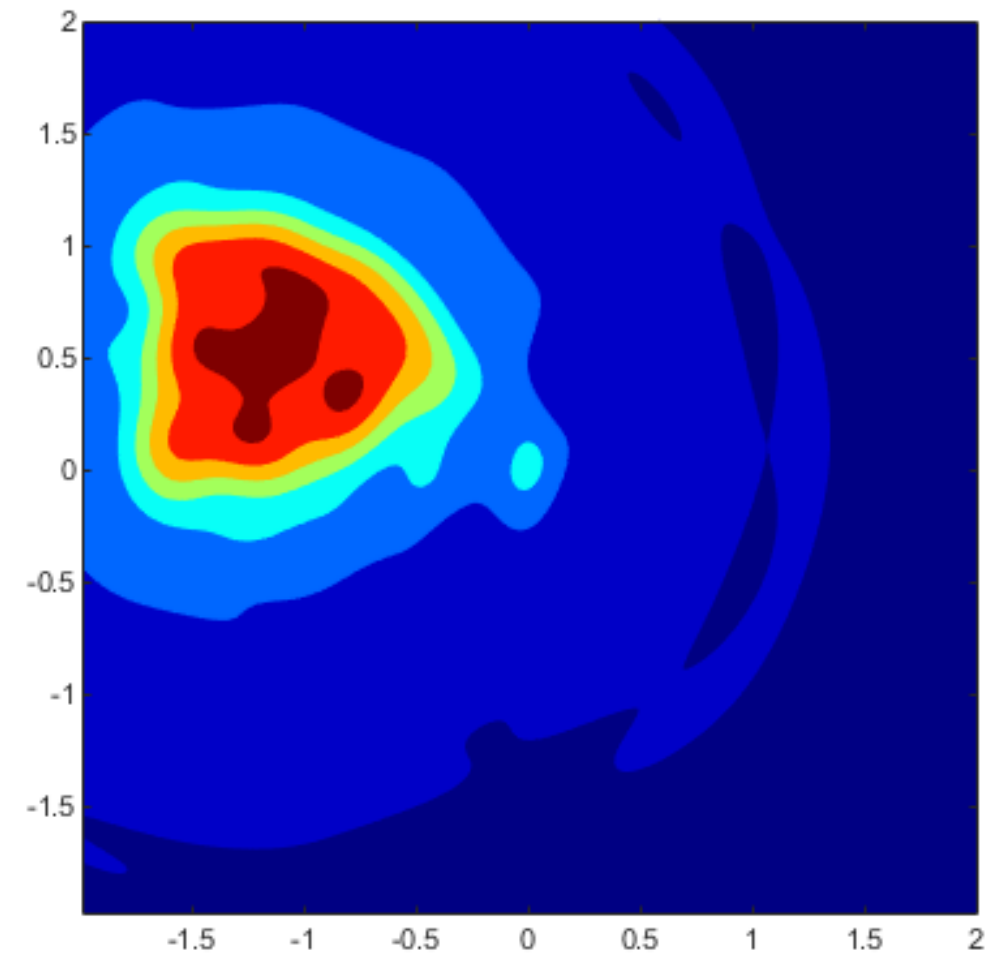}} 
\subfloat[$I_{\text{FF}}(z)\quad \text{with} \quad \{Q_2 ,q_2\}$]{\includegraphics[width=0.32\linewidth]{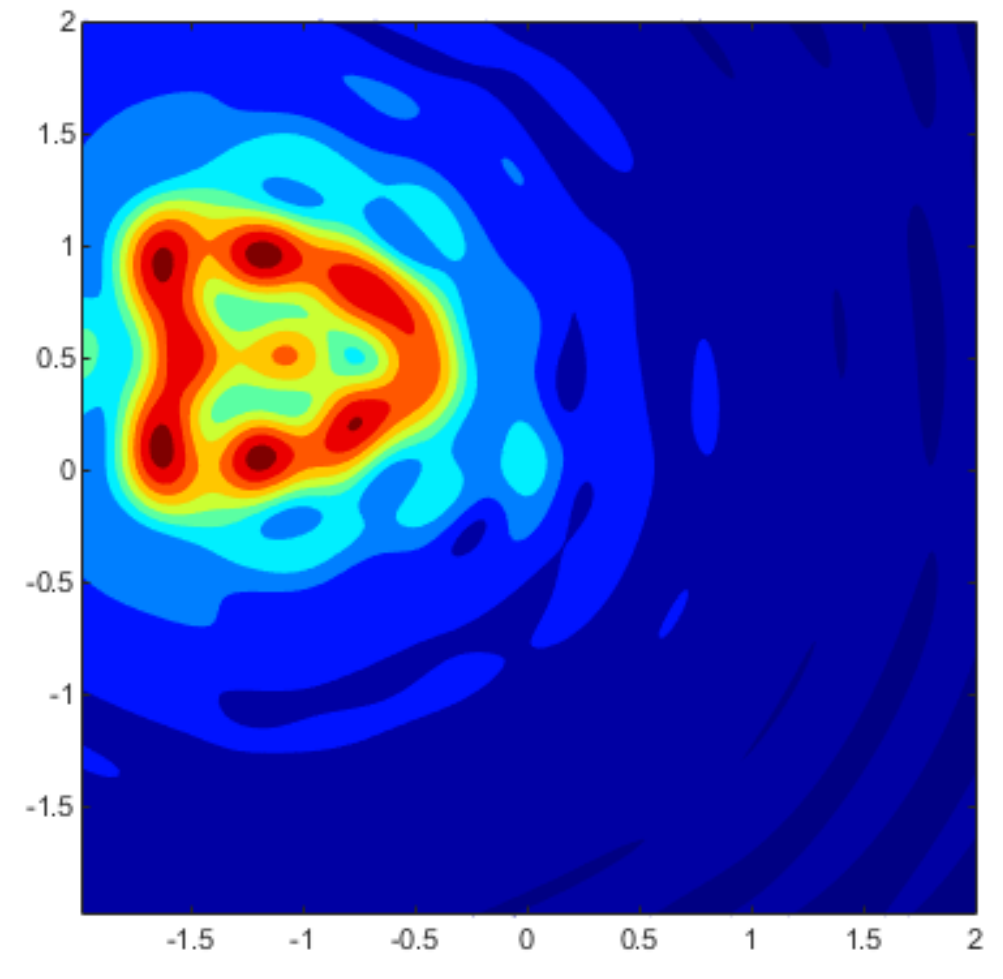}} 
\subfloat[$I_{\text{FF}}(z)\quad \text{with} \quad \{Q_3 ,q_3\}$]{\includegraphics[width=0.32\linewidth]{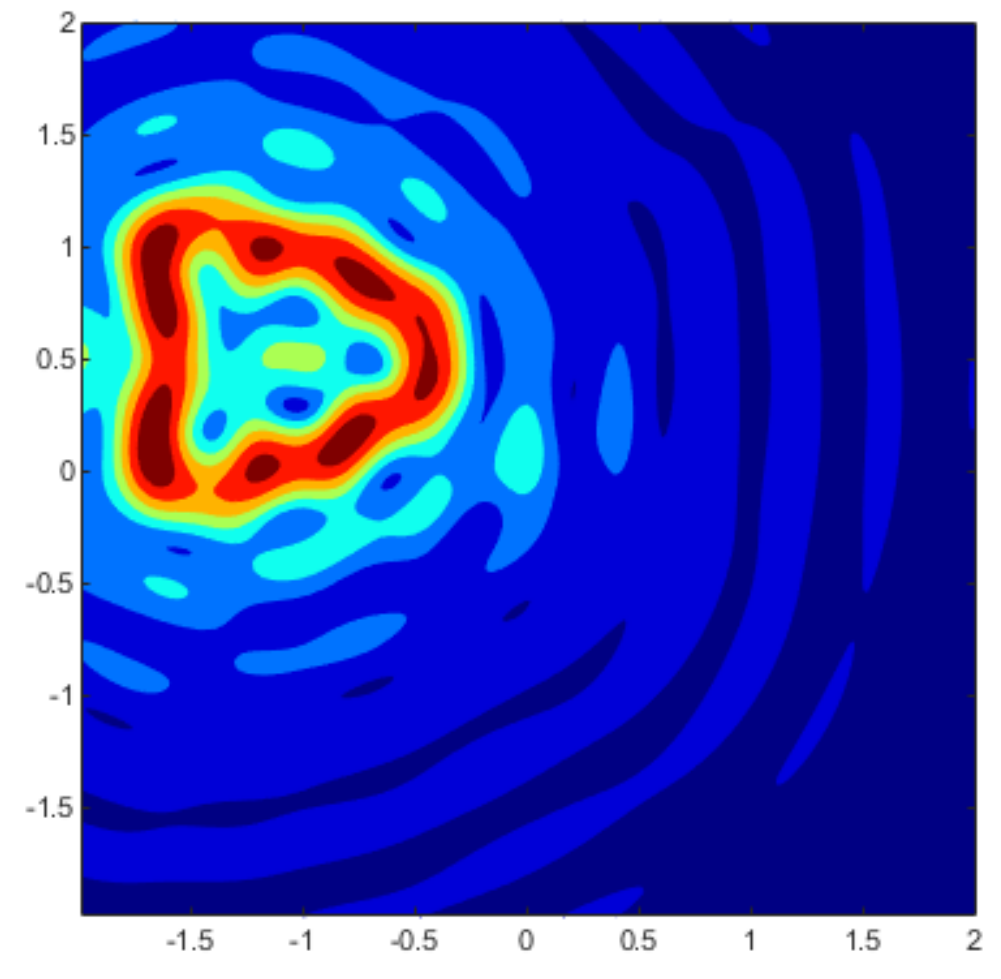}}  
\caption{The reconstruction of the kite-shaped scatterer $D$ where the boundary $\partial D=$kite by the direct sampling method with `Far-Field' transformation.}
\label{recon1FF} 
\end{figure}

\begin{figure}[ht]
\centering
\subfloat[$I_{\text{CD}}(z)\quad \text{with} \quad \{Q_1 ,q_1\}$]{\includegraphics[width=0.32\linewidth]{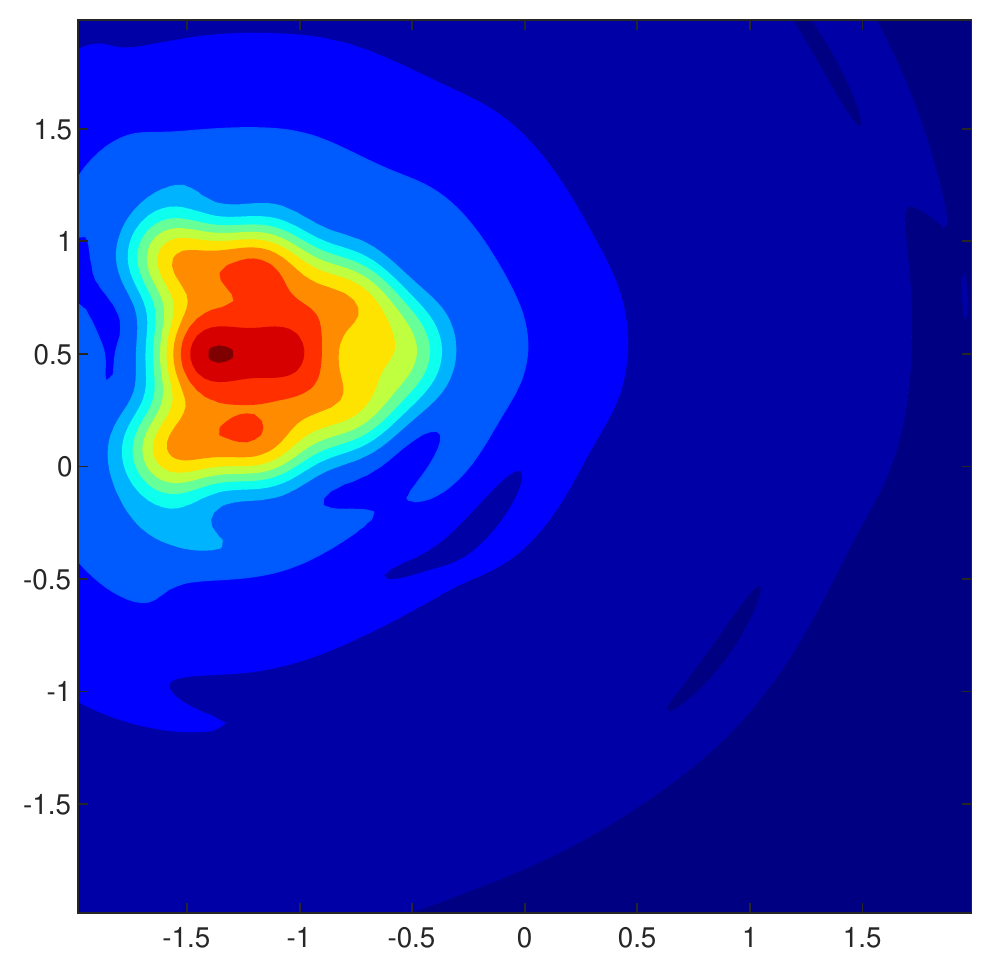}} 
\subfloat[$I_{\text{CD}}(z)\quad \text{with} \quad \{Q_2 ,q_3\}$]{\includegraphics[width=0.32\linewidth]{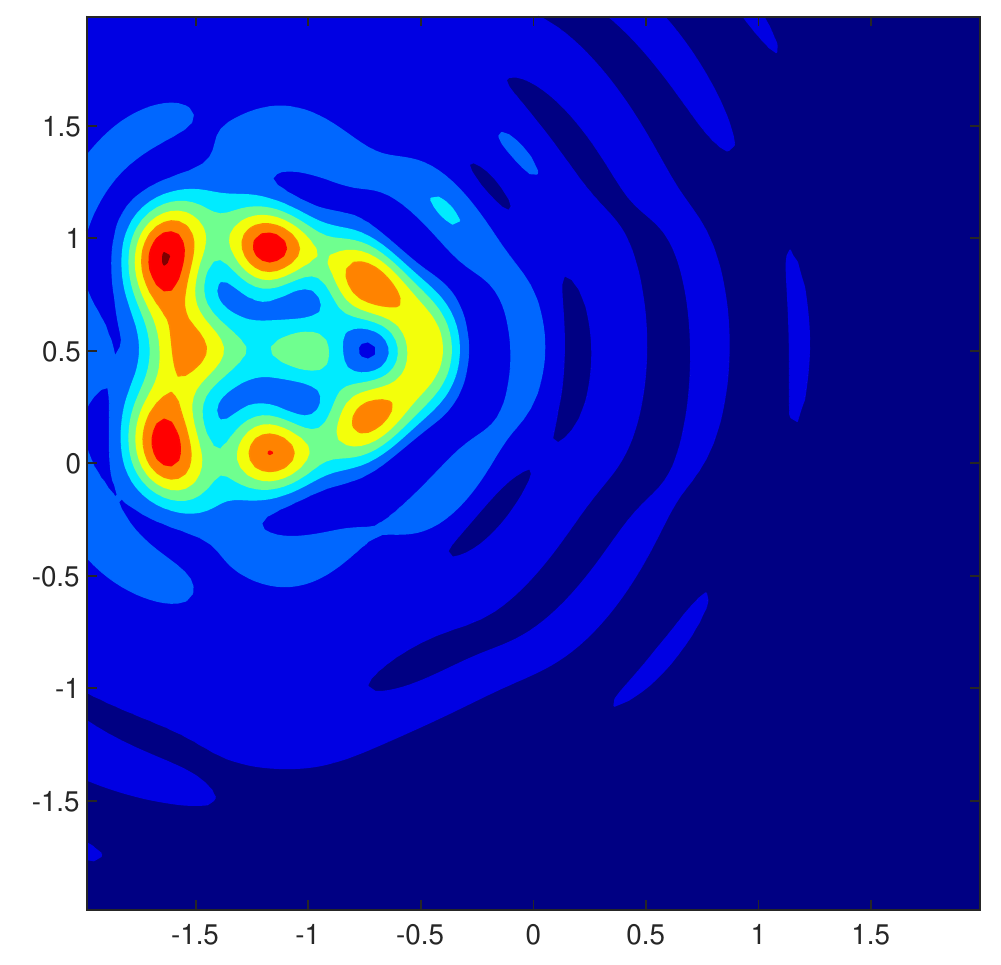}} 
\subfloat[$I_{\text{CD}}(z) \quad \text{with}\quad \{Q_3 ,q_3\}$]{\includegraphics[width=0.32\linewidth]{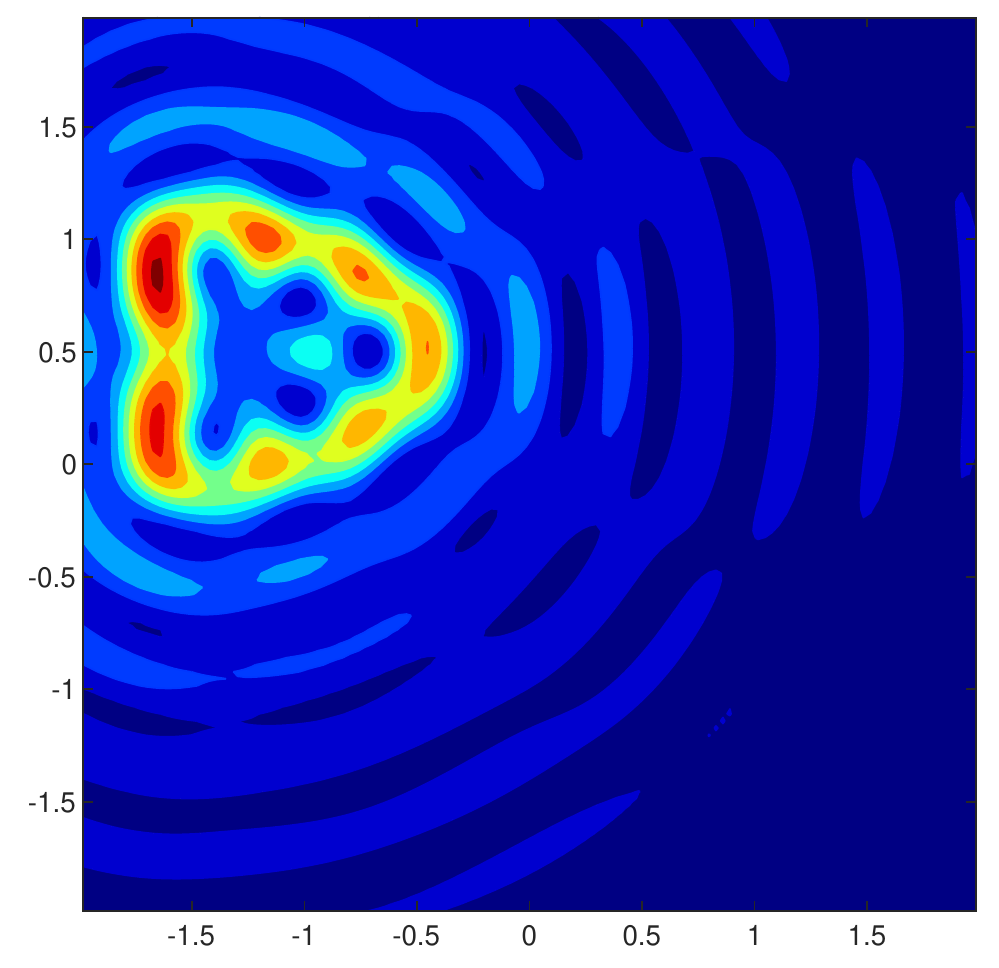}}   
\caption{The reconstruction of the kite-shaped scatterer $D$ where the boundary $\partial D=$kite by the direct sampling method with  Cauchy data. }
\label{recon1CD} 
\end{figure}

\noindent{\bf Example 2: Two Scatterer} \\
Here we consider recovering a scatterer made of two disjoint components. The geometry of the scatterer $D$ is given in Figure \ref{scatterers} (b). Again we present the contour plot the imaging functionals where the boundary is given by $\partial D = \text{disk} \cup \text{rectangle}$. The disk and rectangle are given by 
\begin{equation*}
\begin{array}{ll}
\text{disk} & = \; \{(x_1, x_2) \,\, :  \,\, (x_1 + 1)^2 + (x_2 + 1)^2 = 0.25^2 \} \\[1.5ex]
\text{rectangle} & = \; \{(x_1, x_2)   \,\, : \,\, |x_1 - 0.5| = 0.5, |x_2 + 0.5| = 0.25 \}.
\end{array}
\end{equation*}
The reconstruction for $I_{\text{FF}} (z)$ is given in Figure \ref{recon2FF} and $I_{\text{CD}}(z)$ is given in Figure \ref{recon2CD}, respectively. Here we consider the case of an isotropic and anisotropic scatterer. \\
\begin{figure}[ht]
\centering
\subfloat[$I_{\text{FF}}(z)\quad \text{with} \quad \{Q_1 ,q_1\}$]{\includegraphics[width=0.32\linewidth]{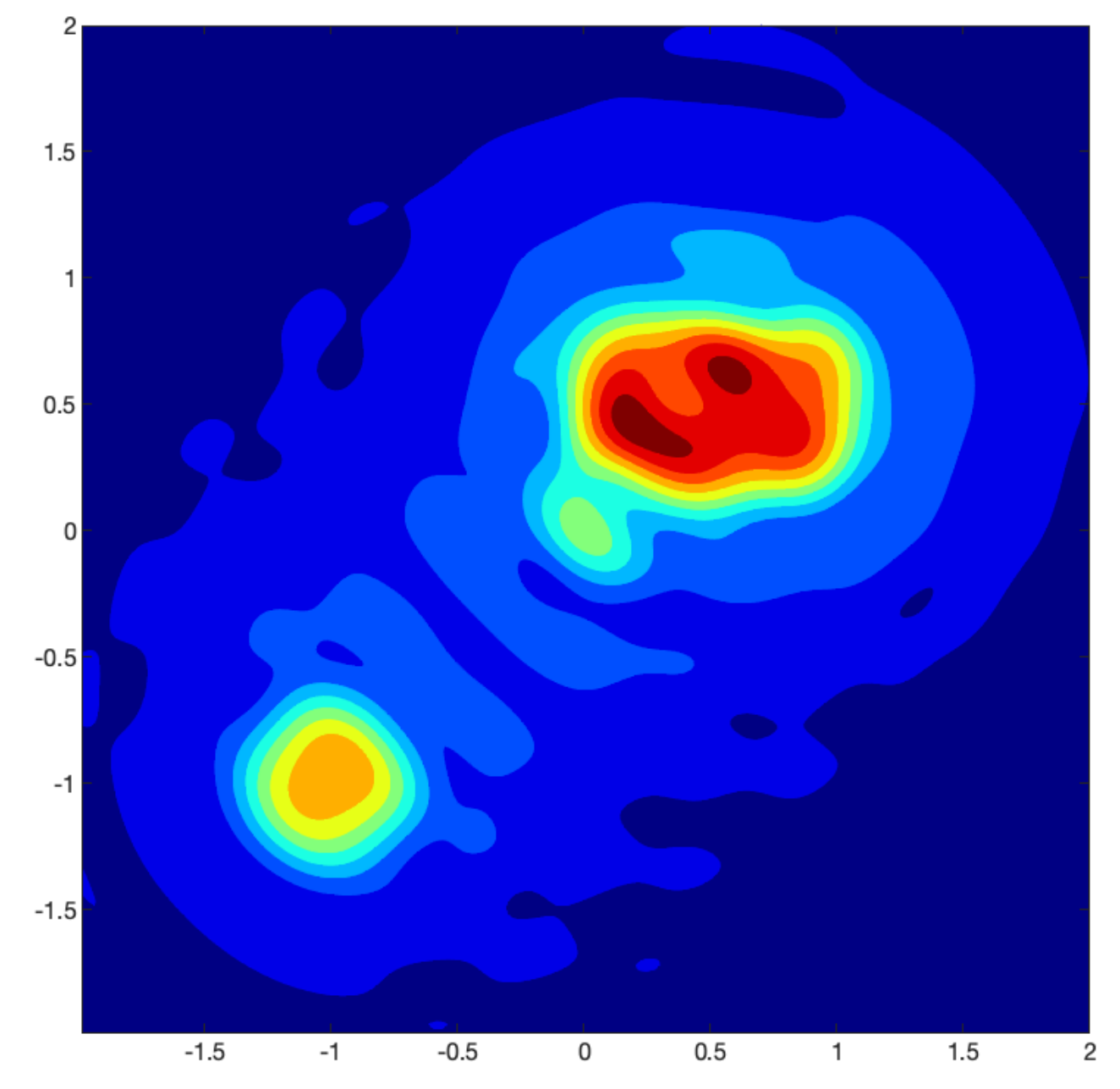}} 
\subfloat[$I_{\text{FF}}(z)\quad \text{with} \quad \{Q_2 ,q_2\}$]{\includegraphics[width=0.32\linewidth]{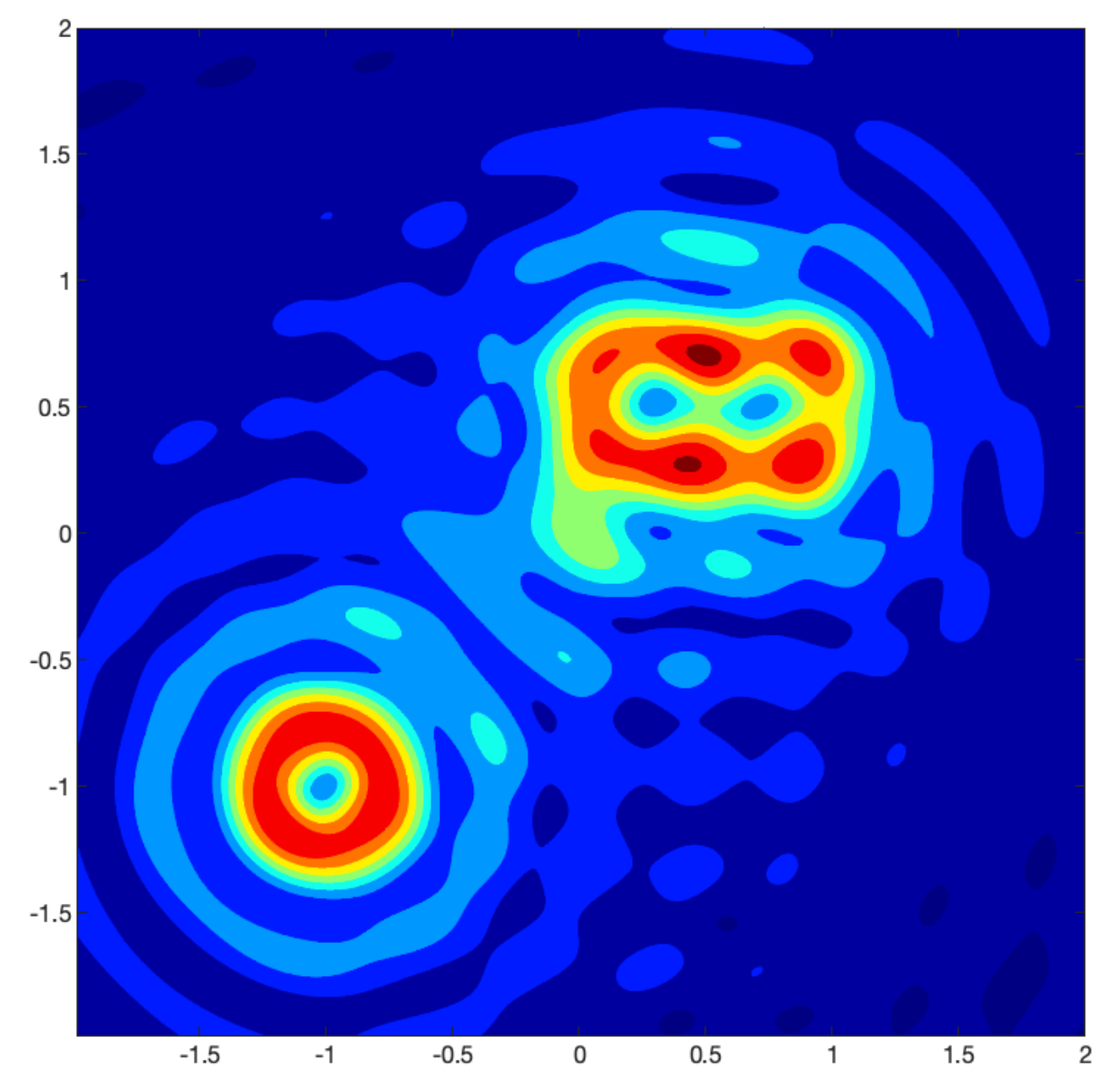}} 
\subfloat[$I_{\text{FF}}(z)\quad \text{with} \quad \{Q_3 ,q_3\}$]{\includegraphics[width=0.32\linewidth]{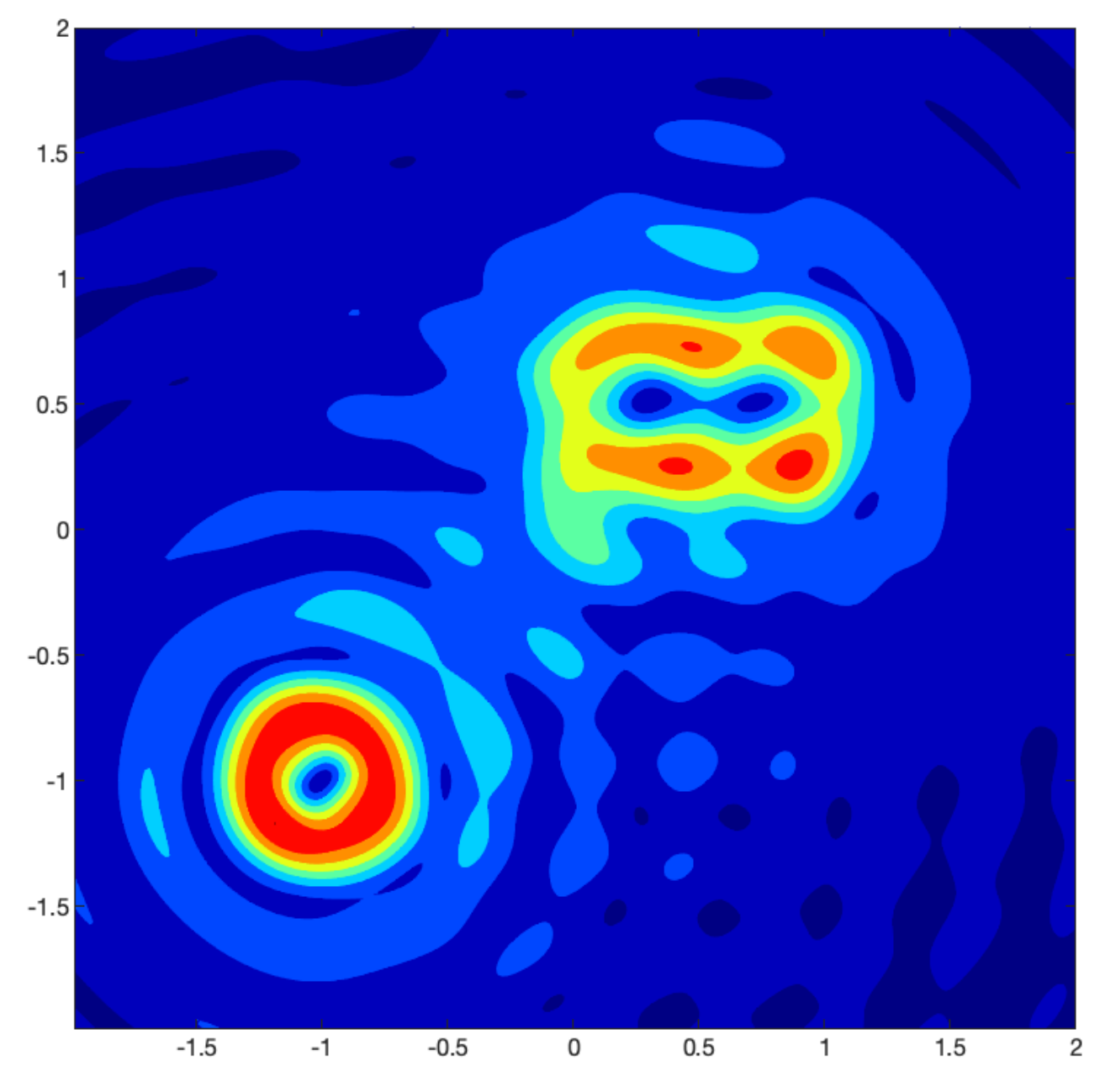}}  
\caption{The reconstruction of the scatterer with two disjoint components $\partial D = \text{disk} \cup \text{rectangle}$ by the direct sampling method with `Far-Field' transformation. }
\label{recon2FF} 
\end{figure}

\begin{figure}[ht]
\centering
\subfloat[$I_{\text{CD}}(z) \quad \text{with} \quad \{Q_1 ,q_1\}$]{\includegraphics[width=0.32\linewidth]{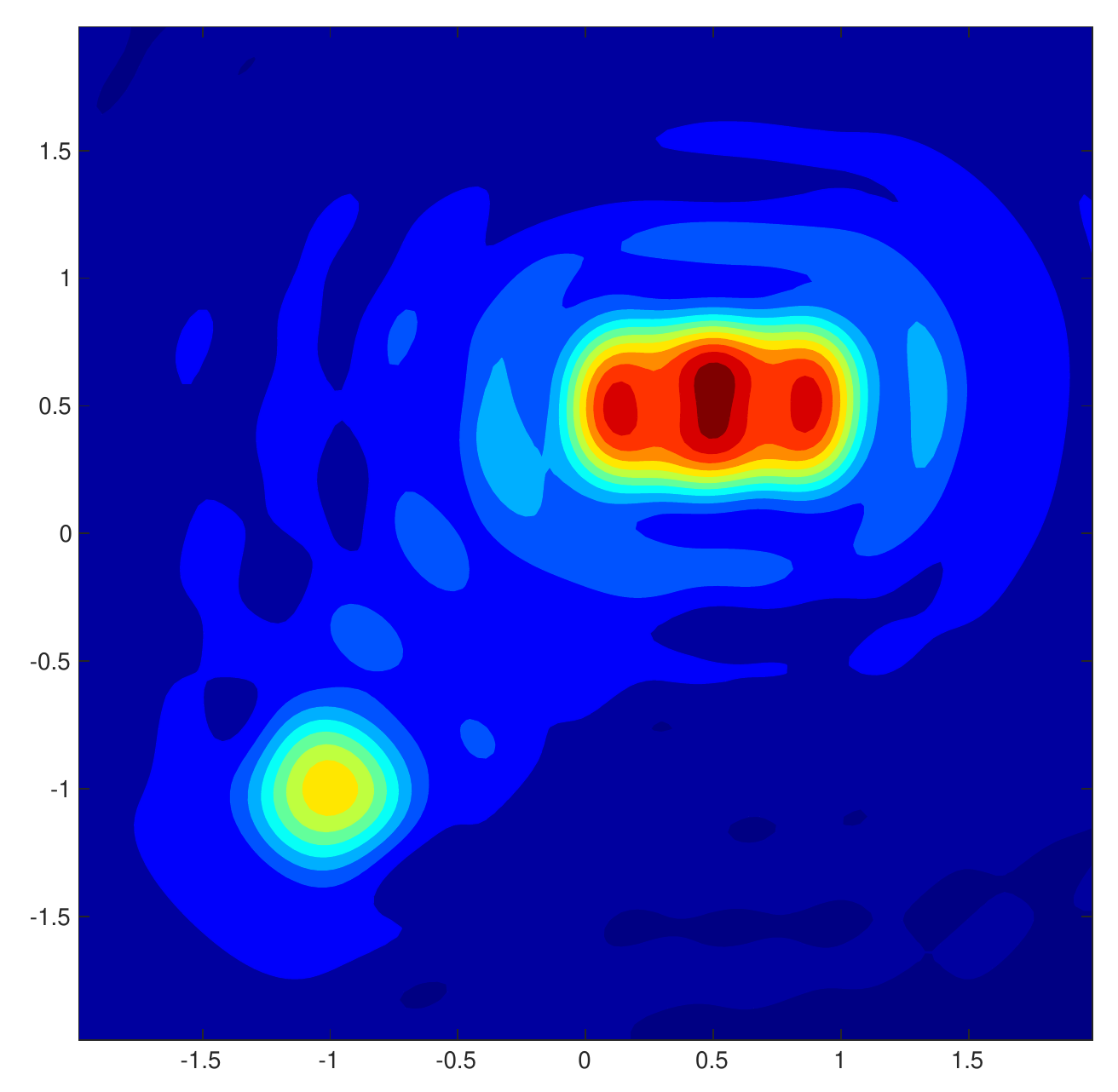}} 
\subfloat[$I_{\text{CD}}(z) \quad \text{with} \quad \{Q_2 ,q_2\}$]{\includegraphics[width=0.32\linewidth]{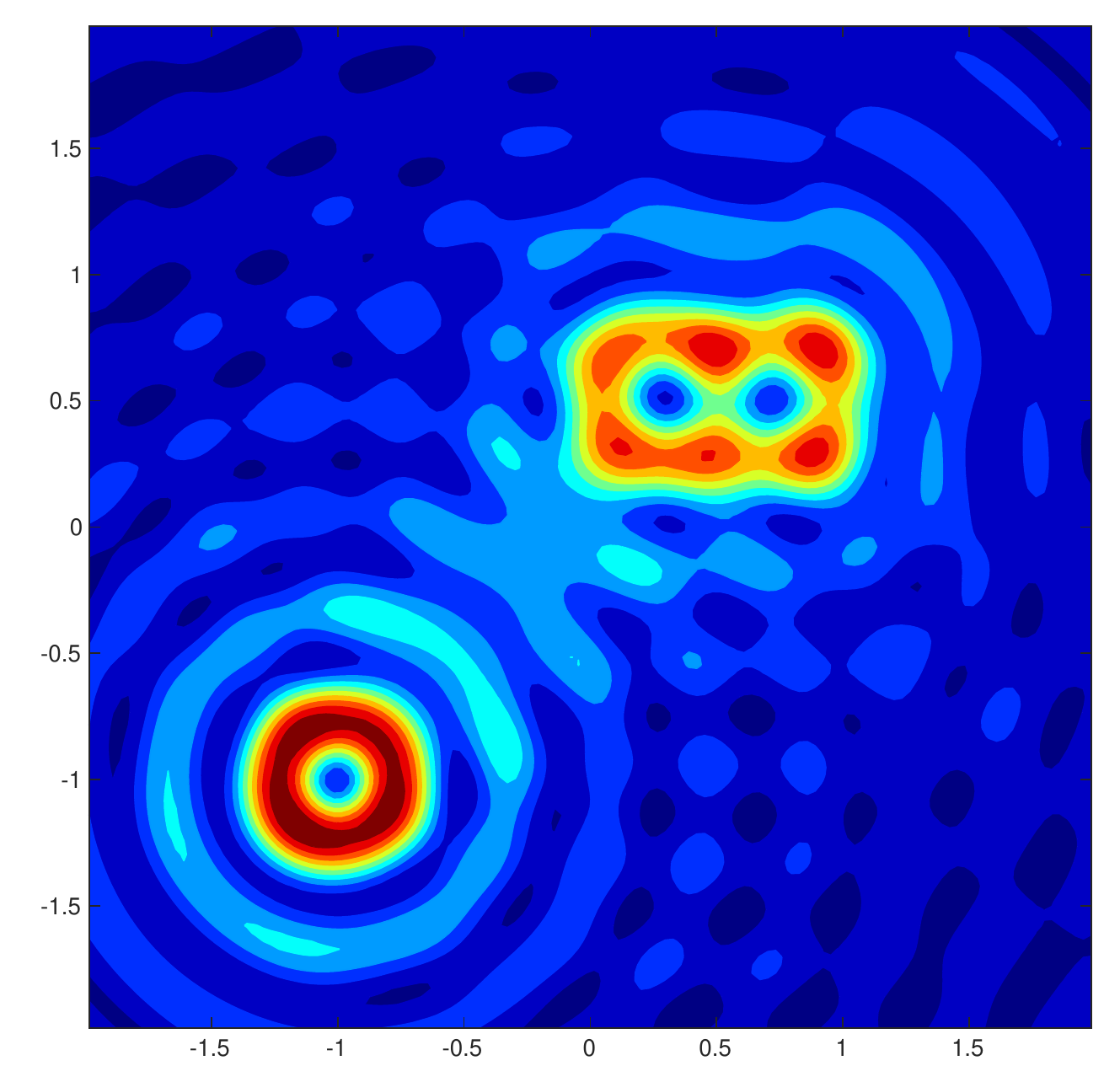}} 
\subfloat[$I_{\text{CD}}(z) \quad \text{with} \quad \{Q_3 ,q_3\}$]{\includegraphics[width=0.32\linewidth]{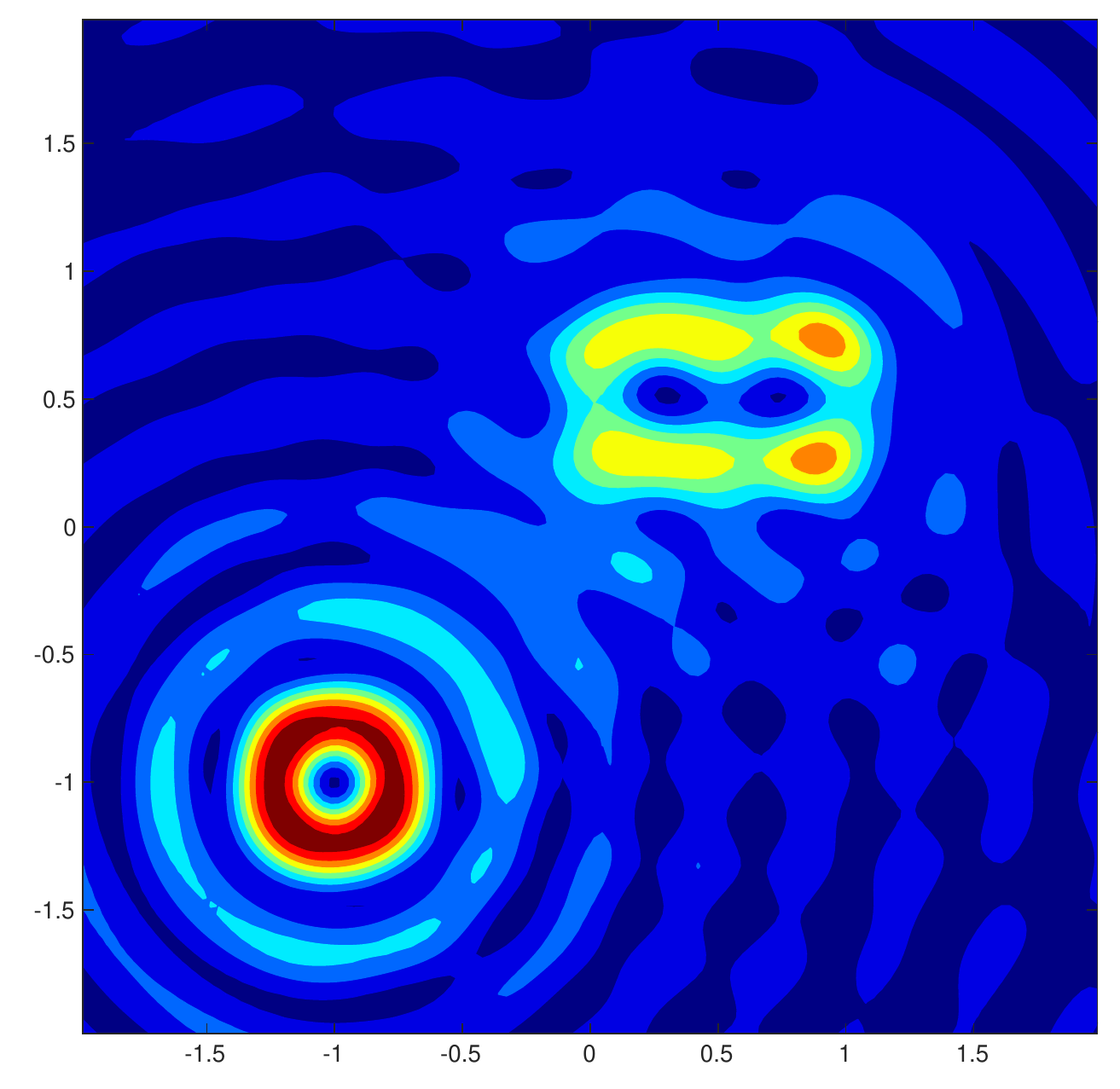}}   
\caption{The reconstruction of the scatterer with two disjoint components $\partial D = \text{disk} \cup \text{rectangle}$ by the direct sampling method with Cauchy data. }
\label{recon2CD} 
\end{figure}

\noindent{\bf Example 3: Three Scatterer} \\
For our last example we will consider recovering the scatterer that is comprised of three disjoint components as depicted in Figure \ref{scatterers} (c). We plot the imaging functionals to recover scatterer $D$ where the boundary $\partial D = \text{kite} \cup \text{ellipse}  \cup \text{peanut} $ where the geometry of the scatterer is given by 
\begin{equation*}
\begin{array}{ll}
\text{kite} & = \; \left\{(x_1, x_2) \, : \, \begin{array}{ll}
x_1 &= -0.75 + 0.2\cos t - 1.84(0.75 + 0.2\sin t)^2, \\
x_2 &= 0.75 + 0.2\sin t,
\end{array} \; 0 \leq t \leq 2\pi\right\} \\[2.5ex]
\text{ellipse} & = \; \{(x_1, x_2)  \, : \, \; x_1 = 1 + 0.25 \cos t \, , \, x_2 = -0.75 + 0.5\sin t , \,  \; 0 \leq t \leq 2\pi \}, \\[2ex]
\text{peanut} &= \left\{(x_1,x_2)  \, : \, \;  \big[(x_1+0.75)^2 + (x_2+1)^2 \big]^2 - 0.32\big[(x_1 + 0.75)^2 - (x_2 + 1)^2\big] = 0.0154 \right\}.
\end{array}
\end{equation*}
The reconstruction for $I_{\text{FF}} (z)$ is given in Figure \ref{recon3FF} and $I_{\text{CD}}(z)$ is given in Figure \ref{recon3CD}, respectively. Here we consider the case of an isotropic and anisotropic scatterer. \\
\begin{figure}[ht]
\centering
\subfloat[$I_{\text{FF}}(z)\quad \text{with} \quad \{Q_1 ,q_1\}$]{\includegraphics[width=0.32\linewidth]{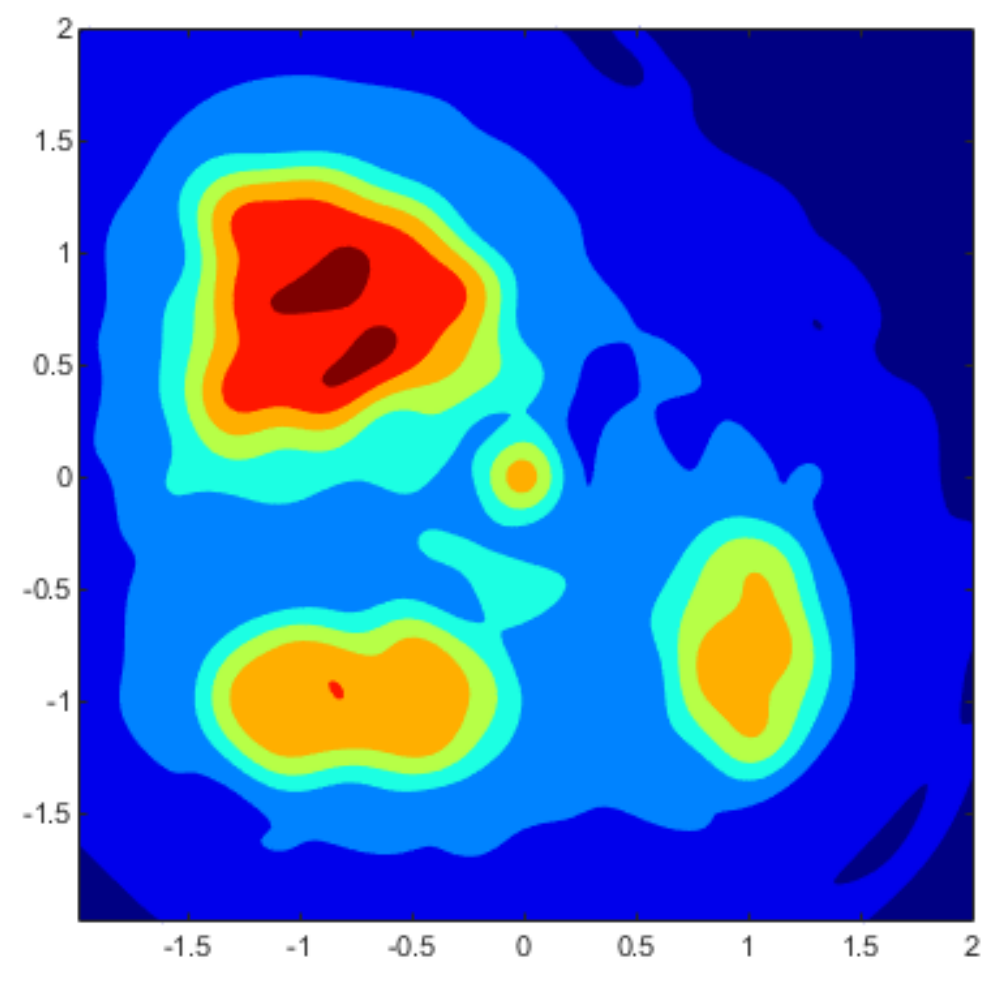}} 
\subfloat[$I_{\text{FF}}(z)\quad \text{with} \quad \{Q_2 ,q_2\}$]{\includegraphics[width=0.32\linewidth]{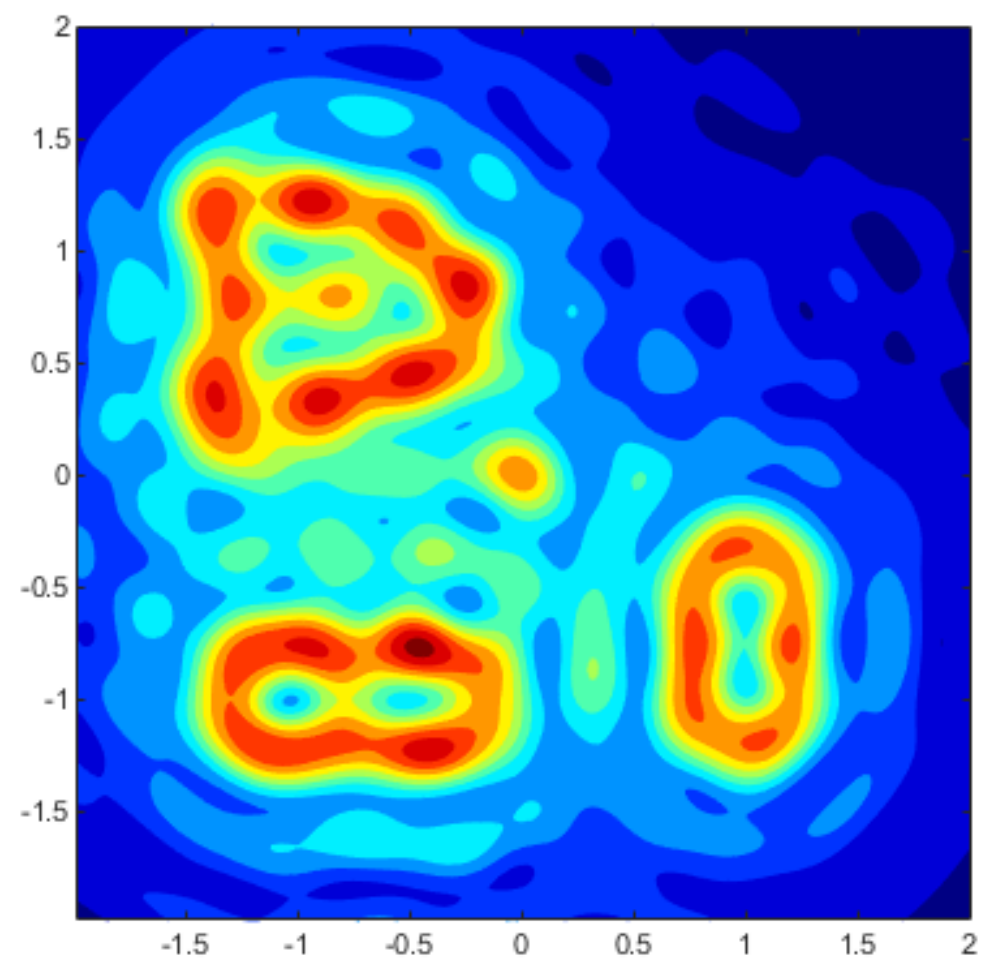}} 
\subfloat[$I_{\text{FF}}(z)\quad \text{with} \quad \{Q_3 ,q_3\}$]{\includegraphics[width=0.32\linewidth]{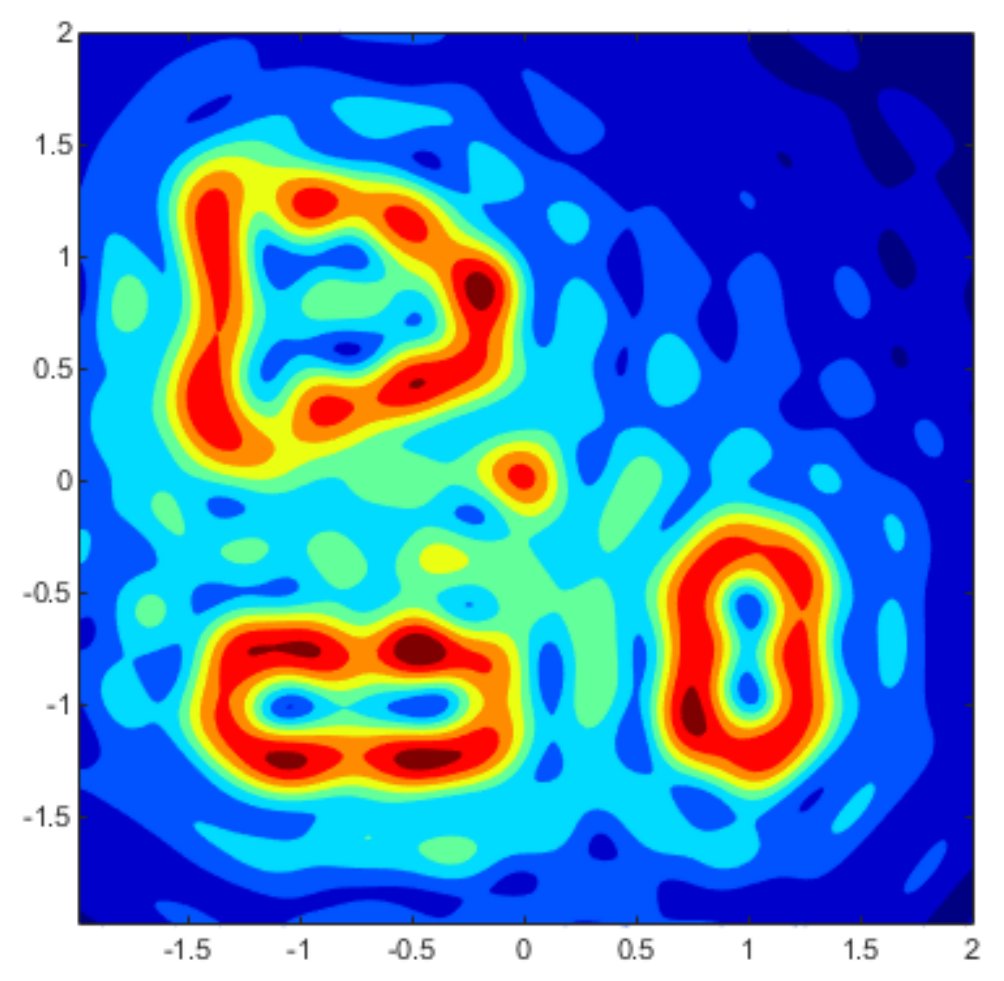}}  
\caption{The reconstruction of the scatterer with three disjoint components $\partial D = \text{kite} \cup \text{ellipse}  \cup \text{peanut} $  by the direct sampling method with `Far-Field' transformation.  }
\label{recon3FF} 
\end{figure}

\begin{figure}[ht]
\centering
\subfloat[$I_{\text{CD}}(z)\quad \text{with} \quad \{Q_1 ,q_1\}$]{\includegraphics[width=0.32\linewidth]{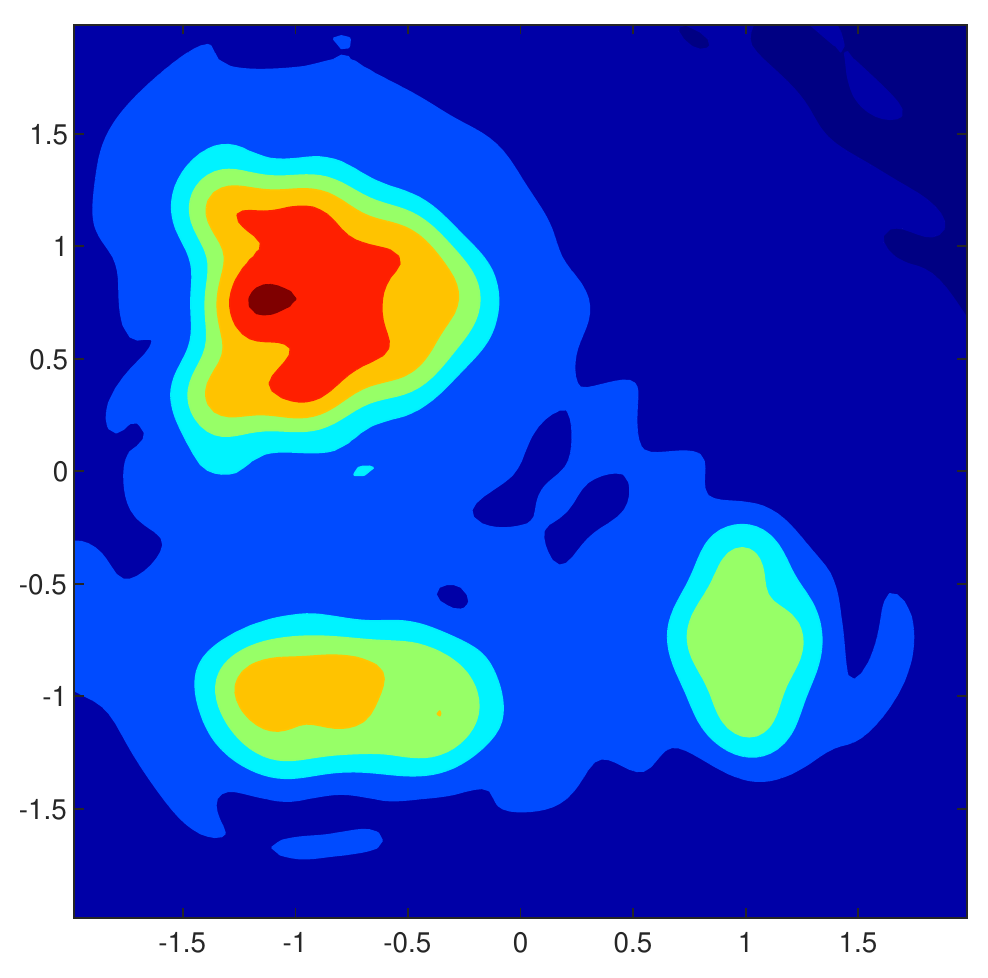}} 
\subfloat[$I_{\text{CD}}(z)\quad \text{with} \quad \{Q_2 ,q_2\}$]{\includegraphics[width=0.32\linewidth]{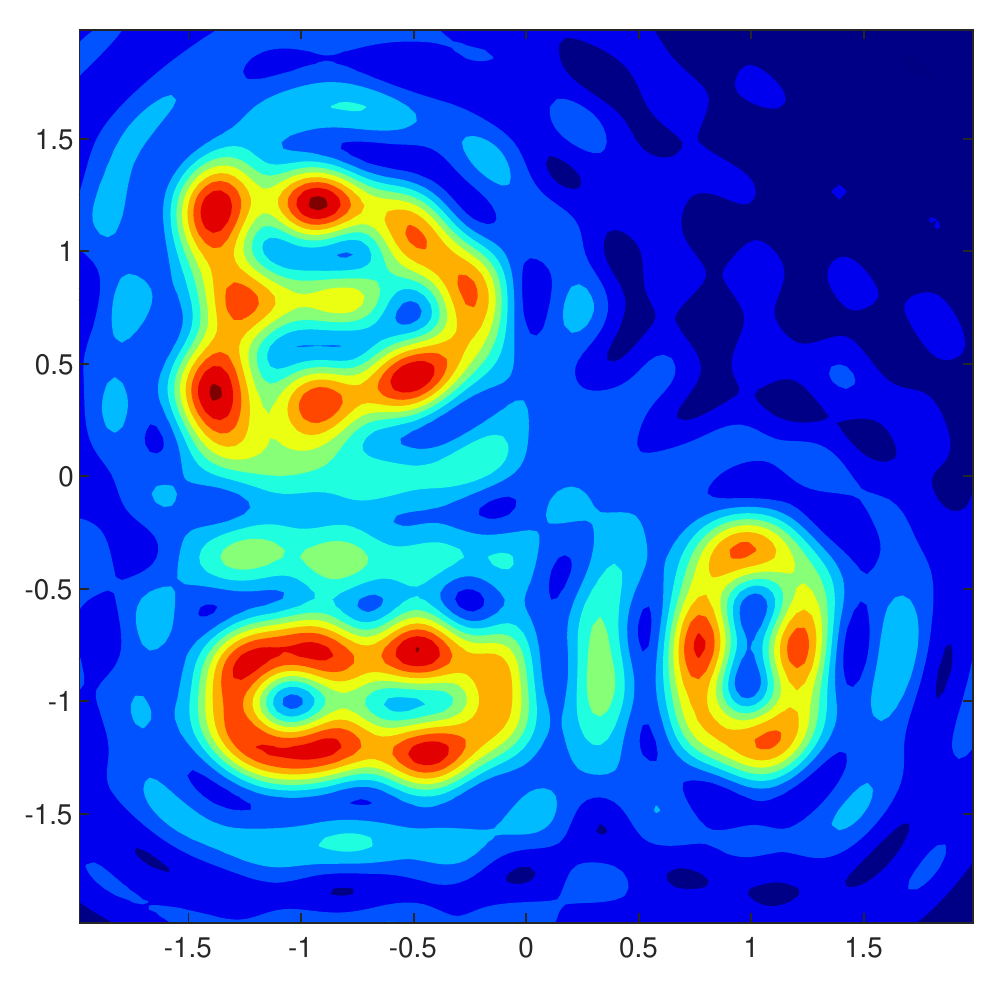}} 
\subfloat[$I_{\text{CD}}(z)\quad \text{with} \quad \{Q_3 ,q_3\}$]{\includegraphics[width=0.32\linewidth]{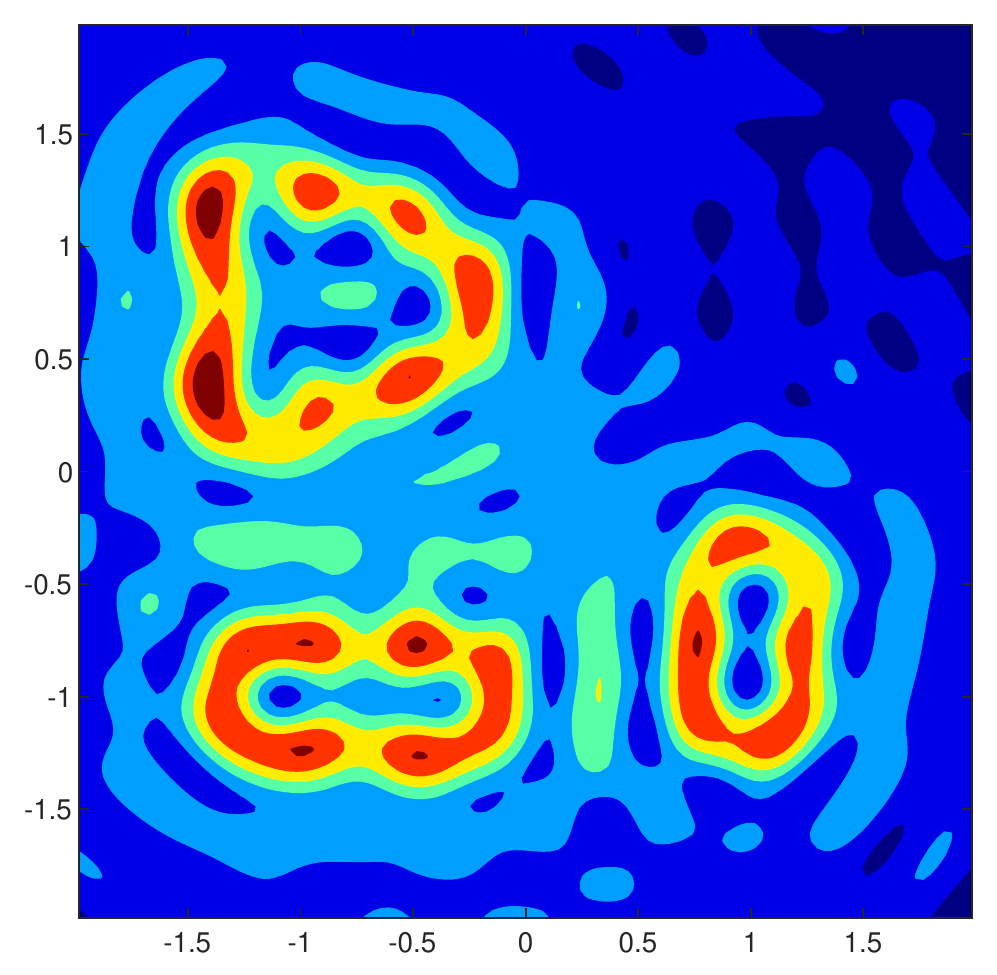}}   
\caption{The reconstruction of the scatterer with two disjoint components $\partial D = \text{kite} \cup \text{ellipse}  \cup \text{peanut}$ by the direct sampling method with Cauchy data.}
\label{recon3CD} 
\end{figure}

Here we see that the two imaging functionals are able to recover both isotropic and anisotropic scatterers. We also see that one can recover multiple scatterers in the presents of a significant amount of noise in the data. The examples given validate the theoretical results given in the previous sections and shows the applicability of these new direct sampling functionals for recovering the region $D$ with little a priori information. \\

\noindent{\bf Acknowledgments:} The research of I. Harris is partially supported by the NSF DMS Grant 2107891. The research of D.-L. Nguyen is partially supported by the NSF DMS Grant 1812693.


\begin{thebibliography}{1}
\bibitem{CCH-book} 
F. Cakoni, D. Colton, and H. Haddar, 
{\it Inverse Scattering Theory and Transmission Eigenvalues} CBMS-NSF Regional Conference Series in Applied Mathematics, SIAM 2016.


\bibitem{Cakon2019}
F.~Cakoni, H.~Haddar, and A.~Lechleiter.
\newblock On the factorization method for a far field inverse scattering in the time domain.
\newblock {\em SIAM J. Math. Anal.}, 51:854--872, 2019.


\bibitem{Chen_2013}
{\sc J.~Chen, Z.~Chen, and G.~Huang}, {\em Reverse time migration for extended obstacles: acoustic waves}, Inverse Problems, 29 (2013), p.~085005.

\bibitem{CK3}
{\sc D.~Colton and R.~Kress}, {\em Inverse acoustic and electromagnetic
  scattering theory}, vol.~93 of Applied Mathematical Sciences, Springer, New
  York, third~ed., 2013.
  
 \bibitem{evans}
\newblock L. Evans, 
\newblock \emph{``Partial Differential Equation''},  
\newblock 2$^{nd}$ edition, AMS 2010.
 
 
\bibitem{heat-fm}
\newblock J. Guo, G. Nakamura and H. Wang,
\newblock The factorization method for recovering cavities in a heat conductor
\newblock preprint (2019)  arXiv:1912.11590.
 
  
\bibitem{Harri2019}
I.~Harris and A.~Kleefeld.
\newblock Analysis of new direct sampling indicators for far-field measurements.
\newblock {\em Inverse Problems}, 35:054002, 2019.

\bibitem{HarrisNguyen2020}
I.~Harris and D-L.~Nguyen.
\newblock Orthogonality Sampling Method for the Electromagnetic Inverse Scattering Problem.
\newblock {\em SIAM  Journal on Scientific Computing}, 42:3  B722--B737, 2020.

\bibitem{Harris-Rome}
\newblock I. Harris and S. Rome, 
\newblock Near field imaging of small isotropic and extended anisotropic scatterers, 
\newblock Applicable Analysis, {\bf 96} issue 10 (2017) , 1713-1736.

\bibitem{not-uniq} 
\newblock P. H\"{a}hner, 
\newblock On the uniqueness of the shape of a penetrable, anisotropic obstacle, 
\newblock J. Comput. Appl. Math. 116 (2000)167--180

\bibitem{nf-fm-isotropic} G. Hu, J. Yang, B. Zhang and H. Zhang, 
\newblock {Near-field imaging of scattering obstacles with the factorization method}
\newblock {\it Inverse Problems} {\bf 30} (2014)  095005.

\bibitem{Ito2012}
K.~Ito, B.~Jin, and J.~Zou.
\newblock A direct sampling method to an inverse medium scattering problem.
\newblock {\em Inverse Problems}, 28:025003, 2012.

\bibitem{Ito2013}
K.~Ito, B.~Jin, and J.~Zou.
\newblock A direct sampling method for inverse electromagnetic medium scattering.
\newblock {\em Inverse Problems}, 29:095018, 2013.



\bibitem{dsm-limap}
S. Kang, et al. 
\newblock Single- and Multi-Frequency Direct Sampling Methods in a Limited-Aperture Inverse Scattering Problem
\newblock {\it IEEE Access} 8:121637--121649, 2020


\bibitem{fm-paper}
A. Kirsch. Characterization of the shape of a scattering obstacle using the spectral data of the far field operator. {\it Inverse Problems} 14:1489, 1998

\bibitem{Kirsc2004}
A.~Kirsch.
\newblock The factorization method for {M}axwell's equations.
\newblock {\em Inverse Problems}, 20:S117--S134, 2004.


\bibitem{kirschbook} 
A. Kirsch and N. Grinberg, 
\newblock \emph{The Factorization Method for Inverse Problems}. 
\newblock Oxford University Press, Oxford 2008.


\bibitem{numericspaper}
 A. Lechleiter and D.-L.Nguyen. A trigonometric Galerkin method for volume integral equations arising in TM grating scattering, {\it Advances in Computational Mathematics} {\bf 40} (2014), 1--25.
 
 
\bibitem{dsm-fm} 
K. H. Leem, J. Liu and G. Pelekanos, 
\newblock Two direct factorization methods for inverse scattering problems, 
\newblock {\it Inverse Problems} {\bf 34} 125004 (2018).


\bibitem{Liu2017}
X.~Liu.
\newblock A novel sampling method for multiple multiscale targets from scattering amplitudes at a fixed frequency.
\newblock {\em Inverse Problems}, 33:085011, 2017.


\bibitem{dsm-dr} 
X. Liu and J. Sun, 
\newblock Data recovery in inverse scattering: From limited-aperture to full-aperture, 
\newblock {\it Journal of Computational Physics}, {\bf 386} (2019) 350--364.

\bibitem{McLean}
W. McLean,{\it``Strongly elliptic systems and boundary integral equations''}. Cambridge: Cambridge University Press 2000.

\bibitem{Nguye2019}
D.-L. Nguyen.
\newblock Direct and inverse electromagnetic scattering problems for
  bi-anisotropic media.
\newblock {\em Inverse Problems}, 35:124001, 2019.

\bibitem{Potth2010}
R.~Potthast.
\newblock A study on orthogonality sampling.
\newblock {\em Inverse Problems}, 26:074015, 2010.

\bibitem{Shixu}
S.~Meng, H.~Haddar, and F.~Cakoni. 
\newblock The factorization method for a cavity in an inhomogeneous medium.
\newblock {\em Inverse Problems}, 30:045008, 2014.


\bibitem{bessel-webpage}
 NIST, 
\newblock ``Asymptotic Expansions for Large Order.'' Digital Library of Mathematical Functions. 
\newblock https://dlmf.nist.gov/10.19
\end{thebibliography}
\end{document}